\documentclass[11pt,letterpaper,reqno]{amsart}

\usepackage{amsmath,amscd}

\usepackage{setspace}
\usepackage{multicol}
\usepackage[lite]{amsrefs} 
\usepackage{amssymb} 
\usepackage{graphicx} 
\usepackage{pdflscape}
\usepackage{euscript}
\usepackage{color}
\usepackage{array}
\usepackage{picture}
\usepackage{epic}
\usepackage{tikz}
 \usetikzlibrary{backgrounds,cd}
\usepackage{comment}
\usepackage{enumerate}
\usepackage{hyperref}
\usepackage[margin=0.97in]{geometry}
\usepackage{caption}
\usepackage{MnSymbol}
\usepackage{enumitem}
\usepackage{fancyhdr} 
\usepackage{centernot}
\usepackage{mathtools}
\usepackage{stmaryrd}
\usepackage{float}
\newcommand{\coker}{\operatorname{coker}}

\newcommand{\set}[1]{\left\{#1\right\}}

\newcommand{\cC}{\mathcal C}
\newcommand{\cF}{\mathcal F}
\newcommand{\cN}{\mathcal N}
\newcommand{\cI}{\mathcal I}

\newcommand{\fa}{{\mathfrak a}}             
\newcommand{\fb}{{\mathfrak b}}

\newcommand{\fg}{{\mathfrak g}}

\newcommand{\fk}{{\mathfrak k}}
\newcommand{\fl}{{\mathfrak l}}

\newcommand{\fp}{{\mathfrak p}}
\newcommand{\fq}{{\mathfrak q}}
\newcommand{\fr}{{\mathfrak r}}
\newcommand{\ft}{{\mathfrak t}}

\newcommand{\IC}{\cI \cC}
\newcommand{\CC}{\cC \cC}

\newcommand{\ga}{\alpha}
\newcommand{\gre}{\epsilon}

\numberwithin{equation}{section}

\theoremstyle{plain} 
\newtheorem{theorem}{Theorem}[section]
\newtheorem{corollary}[theorem]{Corollary}
\newtheorem{lemma}[theorem]{Lemma}
\newtheorem{proposition}[theorem]{Proposition}
\newtheorem*{proposition*}{Proposition}

\newtheorem{conjecture}[theorem]{Conjecture}

\theoremstyle{definition}
\newtheorem{definition}[theorem]{Definition}

\newtheorem{example}[theorem]{Example}
\newtheorem*{ack}{Acknowledgement}
\theoremstyle{remark}
\newtheorem{remark}[theorem]{Remark}

\newcommand{\C}{\mathbb{C}}   
\renewcommand{\P}{\mathbb{P}}

\newcommand{\U}{\mathcal{U}}   

\newcommand{\cq}{\mathcal{Q}} 

\newcommand{\gb}{\beta}

\newcommand{\Ad}{\mathrm{Ad}}

\usepackage{amsmath,amscd}

\makeatletter
\@namedef{subjclassname@2020}{
  \textup{2020} Mathematics Subject Classification}
\makeatother

\hypersetup{bookmarksdepth=2}

\begin{document}

\title[Irreducible Characteristic Cycles]{Irreducible Characteristic Cycles for Orbit Closures of a Symmetric Subgroup}

\author{William Graham}
\address{Department of Mathematics, University of Georgia, Athens GA 30602, USA}
\email{wag@uga.edu}
\author{Minyoung Jeon}
\address{Center for Complex Geometry, Institute for Basic Science (IBS), Daejeon, 34126, Republic of Korea}
\email{minyoungjeon@ibs.re.kr}
\author{Scott Joseph Larson}
\address{Department of Mathematics, University of Georgia, Athens GA 30602, USA}
\email{scott.larson@uga.edu}
\thanks{All authors are the first authors.}
\thanks{}

\begin{abstract}
 Let $G = GL(n)$ and $K = GL(p) \times GL(q)$ with $p+q=n$, where the groups are taken over $\C$.  In this paper we study a certain family of $K$-orbit closures
on the flag variety $X$ of $G$.  The geometry of these orbit closures plays a central role in the infinite-dimensional 
representation theory of the real Lie group $U(p,q)$, and has applications to degeneracy loci and combinatorics. 
In this paper we use small resolutions to study orbit closures in this family.  We prove that the fibers of these resolutions
are smooth and strongly reduced, as well as a general result that if a variety has a resolution of singularities with these properties, then its characteristic cycle is irreducible.
Hence these orbit closures have irreducible characteristic cycles.  A result of Jones then allows us to calculate the 
torus-equivariant Chern-Mather classes of these orbit closures.
We describe torus fixed points and tangent spaces of the resolutions, and use localization to obtain a formula for these classes.
We conjecture that the Chern-Mather classes of a $K$-orbit closure are equivariantly positive when
expressed in a Schubert basis of equivariant Borel-Moore homology, and use our results to verify the conjecture in
an example.
\end{abstract}

\maketitle 

\thispagestyle{empty}
\setcounter{tocdepth}{1}
\parskip=4pt \baselineskip=14pt

\section{Introduction}

Let $Y$ be a closed irreducible subvariety of a smooth complex algebraic variety $X$, and let \(\IC_Y\) denote the intersection cohomology complex of the trivial local system on the smooth locus of \(Y\).  The characteristic cycle \(\CC_Y\) of \(\IC_Y\) is a Lagrangian cycle on the cotangent bundle $T^*X$, which
carries important geometric information about $Y$.  $\CC_Y$ is said to be irreducible or trivial if it is the
class of the closure of the conormal bundle of the smooth locus of $Y$.  The main results of this paper concern $\CC_Y$
in the case where $Y$ is the closure of a $K$-orbit on the flag variety
$X$ of $G$, for $G = GL(n) \supset K = GL(p) \times GL(q)$ with $p+q=n$.  We prove that if $Y$ belongs to a certain family of $K$-orbit closures, then $\CC_Y$ is irreducible.
We do this by considering a particular resolution of singularities $\mu: Z \to Y$, which Larson proved is small (see \cite{Lar}).  We prove that the fibers
of $\mu$ are smooth and strongly reduced (see Theorem \ref{t:fiberhomog}).
In Section \ref{s:small-charcycle}, we prove a general result (Theorem \ref{t:char-cycle}) that shows that these conditions
on the resolution imply that the characteristic cycle is irreducible.  We conclude that for $Y$ in this family of $K$-orbit closures, $\CC_Y$ is irreducible.

Our main application is to the $T$-equivariant Chern-Mather class $c_M^T(Y)$ of $Y$, where
$T$ is a maximal torus of $K$.  Jones proved\footnote{Jones's result was stated
non-equivariantly, but it extends immediately to the equivariant setting.} that if
that $\mu: Z \to Y$ is a $T$-equivariant small resolution and $\CC_Y$ is irreducible, $c_M^T(Y) = \mu_*^T(c^T(Z) \cap [Z]_T)$.  
The torus $T$ acts on $Z$ with finitely many fixed points.  We describe the $T$-weights of $T_z(Z)$ at each $T$-fixed
point $z$.  Localization in equivariant Borel-Moore homology then allows us to compute $c_M^T(Y)$ in
the $T$-equivariant Borel-Moore homology $H^T_*(X)$.  We conjecture that these classes have a positivity
property analogous to the positivity for equivariant fundamental classes proved in \cite[Theorem 3.2]{Gra}.

More generally, suppose that $G$ is a reductive algebraic group with an involution $\theta$ and let $K = G^{\theta}_0$
be the identity component of the fixed group of $\theta$.  Let $B \supset H$ be a $\theta$-stable Borel subgroup
and maximal torus of $G$, and let $T = (H \cap K)_0$.  Choose the positive system of roots such that
the root spaces in $\mbox{Lie }B$ correspond to positive roots.  Let $W$ be the Weyl group of $G$, and
let $Y_w = \overline{B w B}/B \subset X = G/B$.  The fundamental classes $[Y_w]_T$ form a basis of $H^T_*(X)$ 
over $H^*_T$, the equivariant cohomology of a point.  We make the following conjecture.

\begin{conjecture}\label{conj:positive}
Let $Y$ be the closure of a $K$-orbit in $X$.  When
\(c_M^T(Y) \in H^T_*(X) \) is expressed in terms of the Schubert basis $[Y_w]_T$,
each coefficient is a sum of monomials in the positive simple roots of $G$, restricted to
$T$, with nonnegative integer coefficients.
\end{conjecture}

We verify this conjecture in an example in Section \ref{ss:1212}.  Note that this conjecture
implies the analogous statement when $K$ is only assumed to satisfy $G^{\theta}_0 \subseteq K \subseteq G^{\theta}$.

The orbit closures in our family are described precisely in Theorem \ref{t:irreducible}.  They
are obtained roughly as follows.  Richardson and Springer (see \cite{RS90}, \cite{RS93})
defined a monoid action on the set of $K$-orbit closures.  The orbit closures in our
family are obtained by starting with a smooth orbit closure and acting, via the monoid action,
with a sequence of commuting simple reflections.  Although this is a very special family,
our results add these orbit closures to the short list of singular varieties whose characteristic
cycles are known to be irreducible.  This list includes Schubert varieties in
Grassmannians (see \cite{BresslerFinkelbergLunts1990}, where the result was proved
using Zelevinsky's resolutions of singularities), but it does not include all Schubert
varieties in the full flag variety (see \cite{KS97} for a non-irreducible example in type $A_8$ or \cite{BST2019} for \(U(p,q)\)).

Many of our results are given in greater generality than the particular family of examples
which is our focus.  We hope these results will be useful in other situations, for example, in studying
other resolutions of $K$-orbit closures.

The contents of the paper are as follows.  Section \ref{s:small-charcycle} proves the irreducibility
of $\CC_Y$ under the hypotheses that $Y$ has a small resolution with smooth and strongly
reduced fibers.  As this paper was being completed, we learned of closely related work
of Aluffi \cite{Aluffi2025} on irreducible characteristic cycles.
Section \ref{s:fiber-gen} contains general results about fibers related to mixed spaces,

For most of the remainder of the paper, $G$ denotes a reductive algebraic group with an involution $\theta$
and fixed point group $G^{\theta}$, and flag variety $X$.  We denote by $K$ 
a subgroup of $G$ such that $G^{\theta}_0 \subseteq K \subset G^{\theta}$.
For our main results, we assume that $G = GL(n)$ and $K = GL(p) \times GL(q)$, but as noted
above, many of our results are more general.  Section \ref{ss:K-orb} recalls some general results about
$K$-orbit closures and their parametrizations, due largely to Richardson and Springer, following
work by Lusztig and Vogan (see \cite{RS90}, \cite{RS93}).  The remainder of
Section \ref{s:background} is devoted to the case of
$G = GL(n)$ and $K = GL(p) \times GL(q)$.  In particular, we review
the combinatorics of clans used to parametrize the orbits, the monoid action, and the
$\tau$-invariant.  We also prove some results related to smooth orbit closures, including a description of
the torus-fixed points.

Section \ref{s:resolution} contains our main results about resolutions.
We recall Larson's result (see \cite{Lar}) characterizing when the resolutions in question are small.
We prove Theorem \ref{t:fiberhomog} on the fibers of the resolutions, which states that they are
smooth (in fact products of copies of $\mathbb{P}^1$), and strongly reduced.  The irreduciblity
of the characteristic cycles follows, as does the fact that the Chern-Mather class of such
a $K$-orbit closure is the pushforward of the total Chern class of the source. Section \ref{ss:isomorphic} is a complement to
these results: it gives an alternative description of the resolution that exhibits more of its symmetry,
and is useful in analyzing the fibers.

Section \ref{ss:mixed} contains a description of tangent spaces to mixed spaces, extending
results from \cite{GZ} to the greater generality needed in this paper.  Section \ref{ss:tangentK} contains a description of the tangent spaces of resolutions
of $K$-orbit closures at $T$-fixed points.  Section \ref{ss:fixedpoint-gl} studies the fibers of 
our family of resolutions over $T$-fixed points.  Section \ref{ss:Chern-Mather} explains how to use
localization to calculate push-forwards of equivariant Chern classes, and how to express these (and hence
the Chern-Mather classes of the orbit closures in our family) in terms of the Schubert basis.
Section \ref{s:examples} contains some examples explaining the realization of these resolutions as
configuration spaces, and demonstrating how these realizations can be used to verify the description of the
fibers in Theorem \ref{t:fiberhomog}.  Section \ref{ss:1212} gives an example expressing the equivariant Chern-Mather
in terms of the Schubert basis, verifying Conjecture \ref{conj:positive} in this example.

\begin{ack}
The third author would like to thank their advisor Roger Zierau for many useful discussions on characteristic cycles. M. Jeon is supported by the Institute for Basic Science (IBS-R032-D1).
This material is based upon work supported by a grant from the Institute for Advanced Study School of Mathematics.
\end{ack}

\section{Small resolutions and characteristic cycles} \label{s:small-charcycle}
In this section we prove a general theorem showing that under certain hypotheses, characteristic cycles
are irreducible.  As noted in the introduction, this theorem is closely related to recent work
of Aluffi (see \cite{Aluffi2025}).  

We begin by recalling some facts about characteristic cycles; see \cite{KS} and \cite{Dim04} for more information.
Suppose $X$ is a smooth variety and $X = \bigcup S_i$ is a Whitney stratification of $X$.  Let $\cF$ be an element of the
derived category of constructible sheaves on $X$, and assume that $\cF$ is constructible with respect to the  stratification $X = \bigcup S_i$.
The  singular support $SS(\cF)$ is a closed conical Lagrangian subset of $T^* X$ whose irreducible components
are closures of certain of the conormal bundles $\overline{T^*_{S_i}X}$ to the strata (see \cite[Section 4.3]{Dim04}, where
$SS(\cF)$ is denoted $CV(\cF)$).
The characteristic cycle $\CC(\cF)$ is a formal linear combination of cycles
of the form $\sum n_i [\overline{T^*_{S_i}X}]$ where $\overline{T^*_{S_i}X} \subseteq SS(\cF)$.
Let $Y = \overline{S_i}$ and $\cF = \IC_Y$.  In this case, $n_i =1$, and
$\CC(\IC_Y) =  [\overline{T^*_{S_i}X}] + \sum n_j [\overline{T^*_{S_j}X}]$, where the sum is over $j$ such
that $S_j \subset Y$.  If $\CC(\IC_Y) = [\overline{T^*_{S_i}X}]$, then we call the chararacteristic cycle of $\IC_Y$
irreducible or trivial.  We remark that if $U_i$ is the nonsingular locus of $\overline{S_i}$, then $U_i \supset S_i$ and
$\overline{T^*_{S_i}X} = \overline{T^*_{U_i}X}$.

\begin{definition}
Suppose that $f: Z \to X$ is a map of smooth varieties.  A smooth fiber $f^{-1}(x)$ is {\it strongly reduced}
if for all $z \in f^{-1}(x)$, we have
$\ker df_z = T_z f^{-1}(x)$.  If $f$ factors as
$Z \stackrel{\mu}{\rightarrow} Y\stackrel{\iota}{\rightarrow} X$, where $\iota$ is a closed embedding, 
we will sometimes abuse notation and use $\mu$ to denote $f = \iota \circ \mu$.
\end{definition}
This condition appears in \cite[Remark 2.12]{BE} (although they do not use the term strongly reduced), where it is
observed that strongly reduced implies reduced.  

\begin{theorem} \label{t:char-cycle}
Suppose we have a diagram \[ 
\begin{tikzcd}
Z \arrow[d,"\mu"]\arrow[dr,"f"]&\\
Y\arrow[r,"\iota",hook]&X.
\end{tikzcd}
\]
where $Z$ and $X$ are smooth varieties, $\iota$ is a closed embedding, and $\mu$ is a small resolution. 
Suppose that all fibers of $f$ are smooth and strongly reduced.  
Then the characteristic cycle $\CC(\mathcal{IC}_Y)$ is trivial.
\end{theorem}

\begin{proof}
Adapting the notation of \cite[(4.3.2)]{KS}, we consider the commutative diagram 
\begin{equation}
\begin{tikzcd}
T^* Z  \arrow[rd,"\pi_Z"'] & \arrow[l,"^{t}f' "']Z\times_X T^* X \arrow[d,"p"]\arrow[r,"f_{\pi} "]&T^* X\arrow[d,"\pi"]\\
& Z\arrow[r,"f"]&X,
\end{tikzcd}
\end{equation}
where the right hand square is Cartesian.  The map $^t f'$ is defined as $^t f'(z, \alpha) = (df_z)^t (\alpha)$ for
$z \in Z$, $x = f(z)$, and $\alpha \in T_x X$.  Let $\Lambda_z = \ker (df_z)^t \subseteq T_x X$, and let $N = \bigcup_{z \in f^{-1}(x)} \Lambda_z$.

As discussed above,  $\CC(\mathcal{IC}_Y)$ is supported on $SS(\IC_Y)$.
We claim that $SS(\IC_Y) \cap \pi^{-1}(x) \subseteq N$.
Since $\mu$ is a small resolution, $\IC_Y = \mu_* \underline{\mathbb{Q}}_Z$, which we identify with $f_*  \underline{\mathbb{Q}}_Z$;
here $ \underline{\mathbb{Q}}_Z$ denotes the constant sheaf on $Z$.
By \cite[Prop.~5.4.4]{KS}, we have
\begin{equation} \label{e:KS}
SS(\IC_Y) = SS(f_* \underline{\mathbb{Q}}_Z) \subseteq f_{\pi} (^{t}f') ^{-1} (SS( \underline{\mathbb{Q}}_Z) ).
\end{equation}
Since $Z$ is smooth, $SS( \underline{\mathbb{Q}}_Z)$ is the $0$-section of $T^*Z$.  Thus,
$$
(^{t}f') ^{-1} (SS( \underline{\mathbb{Q}}_Z) ) = \{ (z, \alpha) \mid \alpha \in \ker (df_z)^t \},
$$
so
$$
f_{\pi} ( (^{t}f') ^{-1} (SS( \underline{\mathbb{Q}}_Z) ) )  \cap \pi^{-1}(x) \subseteq  \bigcup_{z \in f^{-1}(x)} \Lambda_z = N,
$$
proving the claim.

Next, we claim that  $N$ is a constructible set of dimension at most
$ \dim X - \dim Z + 2 \dim f^{-1}(x)$. 
By hypothesis, $f^{-1}(x)$ is smooth, and $\ker df_z = T_z f^{-1}(x)$.  Consider the exact sequence
$$
0 \rightarrow  T_z f^{-1}(x) \rightarrow T_z Z \stackrel{df_z}{\rightarrow} T_x X \rightarrow C \rightarrow 0,
$$
where $C = \coker df_z$.  Then $\Lambda_z = \ker (df_z)^t = C^*$, so
$$
\dim \Lambda_z = \dim C = \dim X - \dim Z + \dim f^{-1}(x).
$$
Let $d = \dim \Lambda_z$, and
form the Cartesian diagram
$$
\begin{CD}
M  @>>> S \\
@VVV @VVV \\
f^{-1}(x) @>{\Lambda}>> \mathrm{Gr}(d, T_x).
\end{CD}
$$
Here $S \to  \mathrm{Gr}(d, T_x)$ is the tautological bundle.  Since
$M$ is a fiber bundle over $f^{-1}(x)$ with $d$-dimensional fibers, 
$\dim M = d+ \dim f^{-1}(x) = \dim X - \dim Z + 2 \dim f^{-1}(x)$.  The natural map $M \to T_x X$ has image $N$,
so $N$ is constructible and $\dim N \leq \dim M$, proving the claim. 
The two claims imply that
\begin{equation} \label{e:dim-est1}
\dim SS(\IC_Y) \cap \pi^{-1}(x) \leq \dim X - \dim Z + 2 \dim f^{-1}(x).
\end{equation}

There exists a Whitney stratification $Y = \bigcup_{i \in I} S_i$ such that
$\mu |_{\mu^{-1}(S_i)}$ is a fiber bundle (see \cite[Sections 1.6-1.7]{GoreskyMacPherson1988}).  (In our applications, $\mu$ will be equivariant with respect
to a group acting with finitely many orbits on $Y$, and we can take the $S_i$ to be these orbits.)
Let $d_i = \dim f^{-1}(x)$ for $x \in S_i$.  Let $S_0$ denote the open stratum, so $d_0 = 0$.  Since
$\mu$ is small, for $i \neq 0$ we have
\begin{equation} \label{e:dim-est2}
2 d_i < \dim Z - \dim S_i.
\end{equation}

Since $\IC_Y$ is constructible with respect to the stratification of $Y$,
 $SS(\IC_Y)$ is contained in the union of the $\overline{T^*_{S_i} X}$ (see
\cite[Prop.~8.4.1]{KS}).  For $x \in S_i$, $T^*_{S_i} X \cap \pi^{-1}(x)$ consists of the
elements of $T_x^* X$ which vanish on $T_x X$, so
$\dim T^*_{S_i} X \cap \pi^{-1}(x) = \dim X - \dim S_i$.  
Therefore, if $T^*_{S_i} X \subset SS(\IC_Y)$,
\eqref{e:dim-est1} implies that
$$
\dim X - \dim S_i \leq \dim X - \dim Z + 2 d_i,
$$
so
$$
\dim Z - \dim S_i \leq 2 d_i.
$$
Comparing with \eqref{e:dim-est2}, we see that this is impossible unless $i = 0$.  We conclude
that $SS(\IC_Y) \subseteq \overline{T^*_{S_0}}(X)$.  Hence $\CC(\IC_Y) = k [\overline{T^*_{S_0}(X)} ]$ for some integer
$k$.  But $S_0$ is the open stratum of $Y$, so $k = 1$.  
\end{proof}

\section{Generalities on fibers} \label{s:fiber-gen}
This section contains some general results (Propositions \ref{p:homogeneous} and \ref{p:fiber-general}) which we will use to study the fibers
of our resolutions.

Let $G$ be an algebraic group.  We say a map of $G$-varieties $\xi: Y \to Z$ has homogeneous fibers if for all  $y\in Y$, the stabilizer group $G^y$ of $y$
 acts transitively on $\xi^{-1}(y)$.

\begin{proposition} \label{p:homogeneous}
Let $G$ be an algebraic group and let $\xi:Z\rightarrow Y$ be a $G$-equivariant map of $G$-varieties. Suppose that $Y$ has finitely many $G$-orbits. Then
$\xi$ has homogeneous fibers if and only if $\#(G\backslash Z)=\#(G\backslash Y)$.
\end{proposition}

\begin{proof}
We have $\#(G\backslash Z)=\#(G\backslash Y)$ if and only if the inverse image of any $G$-orbit in $Y$ is a single $G$-orbit
in $Z$.  The $G$-orbits in $\xi^{-1}(G \cdot y)$ are in bijection with the $G^y$-orbits in $\xi^{-1}(y)$, so $\xi^{-1}(G \cdot y)$ is a single
$G$-orbit if and only if $\xi^{-1}(y)$ is a single $G^y$-orbit.  The result follows.
\end{proof}

\begin{lemma} \label{l:changefiber}
Let $B$ be a linear algebraic group.  Suppose we have a Cartesian diagram
$$
\begin{CD}
\widetilde{M} @>{\widetilde{\mu}}>> \widetilde{X} \\
@VVV @VV{\pi}V \\
M = \widetilde{M}/B @>{\mu}>> X = \widetilde{X}/B,
\end{CD}
$$
where the vertical maps are principal $B$-bundles.  Suppose $\pi(\widetilde{x}) = x$, and let $\widetilde{F} = \widetilde{\mu}^{-1}(\widetilde{x})$,
$F = \mu^{-1}(x)$.  Let $\widetilde{m} \in \widetilde{F}$ map to $m \in F$.  Suppose that $\widetilde{M}$, $M$, $\widetilde{F}$ and $F$ are smooth. Then
$$
T_m F = \mathrm{Ker}(d\mu_m) \Leftrightarrow T_{\widetilde{m}} \widetilde{F} = \mathrm{Ker}(d\widetilde{\mu}_{\widetilde{m}}) .
$$
\end{lemma}

\begin{proof}
We may replace $X$ by an open neighborhood of $x$ over which the bundle is trivial; then our commutative diagram becomes
$$
\begin{CD}
\widetilde{M} = M \times B @>{\widetilde{\mu}}>> \widetilde{X} = X \times B \\
@VVV @VVV \\
M @>{\mu}>> X,
\end{CD}
$$
where $\widetilde{\mu}=\mu\times \mbox{id}$.
For some $b \in B$, we have $\widetilde{x} = (x, b)$  $\widetilde{F} = F \times \{ b \}$, and $\widetilde{m} = (m,b)$.
Under the identification
 $T_{\widetilde{m}} \widetilde{M} = T_m M \times T_b B$, we have
 \begin{equation}\label{eqn:tan-fib}
 T_{\widetilde{m}}\widetilde{F}=T_{(m,b)}(F \times \{ b \})=T_mF\times T_b\{b\}=T_mF\times\{0\}.
\end{equation}
 Also, $d\widetilde{\mu}_{\widetilde{m}}=d\mu_m\times\mbox{id}$,
 so
 \begin{equation}\label{eqn:ker}
\mathrm{Ker}(d\widetilde{\mu}_{\widetilde{m}})= \mathrm{Ker}(d\mu_m) \times \{ 0 \}.
\end{equation}
Comparing \eqref{eqn:tan-fib} and \eqref{eqn:ker} gives the result.
\end{proof}
 
\begin{proposition} \label{p:fiber-general}
Let $G$ be a linear algebraic group, $P$ a closed subgroup of $G$, and $R$ and $B$ closed subgroups of $P$.
Let $M$ be a right $R$-invariant locally closed subvariety of $G$,
and let $Y = MP \subset G$.
Let 
$$
\mu: Z = M \times^R P/B \to Y/B \subset G/B
$$ 
denote the multiplication map, and let $y \in G$.  Then:
\begin{enumerate}
\item $\mu^{-1}(yB/B) \cong M/R \cap yP/R$.
\item If $y \in Y$ and $p \in P$, then $\mu^{-1}(ypB/B) \cong \mu^{-1}(yB/B)$.
\item Every fiber of $\mu$ is isomorphic to a fiber $\mu^{-1}(mB/B)$ for some $m \in M$.  
\item Let $y' = yp$ for $p \in P$, and let $F = \mu^{-1}(yB/B)$, $F' =  \mu^{-1}(y'B/B)$.
Suppose that $M$ and $F$ (hence also $F'$) are smooth.  Then
$$
\mathrm{Ker}(d\mu_{z'})=T_{z'} F' \mbox{ for all } z' \in F' \Leftrightarrow \mathrm{Ker}(d\mu_{z})=T_{z} F \mbox{ for all } z \in F. 
$$
\end{enumerate}
\end{proposition}

\begin{proof}
If $\mu([m, pB/B]) = yB$, then $mp = y b$ for some $b \in B$; as $B \subset P$, we see that $m \in yP$.  Therefore, we can define
$\ga: \mu^{-1}(yB) \to M/R \cap yP/R$ by $\ga([m, pB/B]) = mR/R$.   If $mR/R \in M/R \cap yP/R$,
then $y^{-1}m$, and therefore $m^{-1} y$, are in $P$.  Therefore we can define $\gb: M/R \cap yP/R \to \mu^{-1}(yB)$ by
$\gb(mR/R) = [m, m^{-1} yB/B]$.  The compositions $\ga \circ \gb$ and $\gb \circ \ga$ are each the identity map, so
$\ga$ and $\gb$ are inverse isomorphisms, proving (1).

 We have a Cartesian diagram
\begin{equation} \label{e:fiber-general}
\begin{CD}
M \times^R P @>{\tilde{\mu}}>> Y = MP \\
@VVV    @VVV \\
M \times^R P/B @>{\mu}>> Y/B.
\end{CD}
\end{equation}
For any $p \in P$ and any $y \in Y$, we claim that we have a composition of isomorphisms: 
\begin{equation} \label{e:comp-iso}
\begin{CD}
\mu^{-1}(yB/B) @>>> \tilde{\mu}^{-1}(y) @>>> \tilde{\mu}^{-1}(yp) @>>>  \mu^{-1}(ypB/B).
\end{CD}
\end{equation}
Indeed, the first and third maps exist and are isomorphisms because the diagram is Cartesian.  The middle map is
right multiplication by $p$; it is an isomorphism 
because $\tilde{\mu}$ is right $P$-equivariant.  This proves the claim; (2) follows.  If $y \in Y$, then
$y = mp$ for some $m \in M$, so by (2), $\mu^{-1}(yB/B) \cong \mu^{-1}(mB/B)$, proving (3).

For (4), we apply Lemma \ref{l:changefiber} to the Cartesian diagram \eqref{e:fiber-general}.  
Let $\tilde{F} = \tilde{\mu}^{-1}(y)$ and $\tilde{F}' = \tilde{\mu}^{-1}(y')$.
Then
\begin{align*}
\mathrm{Ker}(d\mu_{z})=T_z F \mbox{  for all  }z \in F & \Leftrightarrow \mathrm{Ker}(d\tilde{\mu}_{\tilde{z}})=T_{\tilde{z}} \tilde{F} \mbox{ for all }\tilde{z} \in \tilde{F} \\
& \Leftrightarrow \mathrm{Ker}(d\tilde{\mu}_{\tilde{z}'})=T_{\tilde{z}'} \tilde{F}' \mbox{ for all }\tilde{z}' \in \tilde{F}' \\
& \Leftrightarrow  \mathrm{Ker}(d\mu_{z'})=T_{z'} F' \mbox{  for all  }z' \in F' .
\end{align*}
Here, the first and third equivalences hold by Lemma \ref{l:changefiber}.  To see the middle equivalence, let
$\rho$ denote the map on $M \times^R P$ or $Y$ given by right multiplication by $p$, and let $z' = zp = \rho(z)$.
Since $\tilde{\mu} \circ \rho = \rho \circ \tilde{\mu}$, the map $d \rho_z$ takes
$\mathrm{Ker}(d\tilde{\mu}_{\tilde{z}})$ to $\mathrm{Ker}(d\tilde{\mu}_{\tilde{z}'})$ and $T_{\tilde{z}} \tilde{F} $
to $T_{\tilde{z}'} \tilde{F}' $.  The middle equivalence follows.  This proves (4).
\end{proof}

In the setting of the previous proposition, we call $m \in M$ a good base point of the map $\mu$.  Part
(3) of the proposition states that every fiber of $\mu$ is isomorphic to a fiber over a good base point.  Note that
the notion of good base point depends on the particular expression for $Z$.  Section \ref{ss:isomorphic}
contains examples where $Z = M \times^R P/B \cong M' \times^{R'} P'/B$.  In this case, the set of good base points is $M$ or $M'$,
depending on which expression for $Z$ is being considered.

\begin{remark}
In \eqref{e:comp-iso},
the fiber $\mu^{-1}(ypB/B)$ depends on the element
$y$ and not merely on the coset $yB/B$. 
\end{remark}

\section{K-orbit closures} \label{s:background}
In this section, we recall some general results about $K$-orbits in the flag variety $G/B$.  In particular, we recall the parametrization of these orbits, and
 root types with respect to a $K$-orbit.  Then we describe these results concretely
in the case where $G = GL(n)$ and $K = GL(p) \times GL(q)$ where $p+q=n$.  Our main references for this section are \cite{Lar,RS90,RS93,Vog3}.

\subsection{Definitions and notation} \label{ss:definitions}
Let $G$ be a connected complex reductive group with an involution $\theta$.  The involution $\theta$ on $G$ defines an involution on $\mathfrak{g} = \mbox{Lie } G$ which we also denote by $\theta$.  Let $K$ be a closed subgroup of $G$ satisfying
$G^{\theta}_0 \subseteq K \subseteq G^{\theta}$ where $G^{\theta}_0$ is the identity component of  \(G^{\theta}\); then $\fk = \mbox{Lie }K = \fg^{\theta}$.
Let $\fp$ denote the $-1$-eigenspace of $\theta$ in $\fg$.
The group $T$ is regular in $G$, which
means that its centralizer $H$ is a maximal torus of $G$ (cf.\ \cite[2.3]{Brion1999}).   If $V$ is a representation of $T$ (resp.~$H$),
we let $\Phi(V)$ (resp.~$\Phi(V,H)$) denote the multiset of $T$-weights (resp.~$H$-weights) of $V$, counted with multiplicity.
For $G = GL(n)$ and $K = GL(p) \times GL(q)$, we have $H = T$.

Let $W = N_G(H)/H$.  Write $\dot{w}$ for a representative in $N_G(H)$ of $w \in W$.  We will frequently abuse notation
and write $w$ in place of $\dot{w}$.
Given $g \in G$ and a subspace $\fa$ of $\fg$, write ${}^g \fa = (\Ad \ g) (\fa)$.  If $\fa$ is $H$-stable and $x \in W$,  write
${}^x \fa = {}^{\dot{x}} \fa$.

Let $B$ be a $\theta$-stable 
Borel subgroup containing $H$; then $B^{\theta}_0$ is a Borel subgroup of $K$.  
Let $X = G/B$.  Because $T$ is regular, we have $X^T = X^H$
(see \cite{Brion1999}, \cite{Wys2016}); the analogous equality holds for partial
flag varieties as well. Let $Y_w=\overline{BwB}/B$ denote the Schubert variety in $X$ associated to $w\in W$.

 Let $\Phi = \Phi(G,H)$ denote the set of roots of $G$ relative to $H$, with positive system $\Phi^+$
 corresponding to the root spaces in $\fb = \mbox{Lie }B$.
Let $\Delta = \{ \ga_1, \ldots, \ga_\ell \}$ denote the corresponding set of simple roots.  Let
$s_{\ga} \in W$ denote the reflection corresponding to $\ga \in \Phi$.  Write $s_i = s_{\ga_i}$, and let
 $S = \{ s_1, \ldots, s_\ell \}$ be the set of simple reflections.  If $\ga \in \Phi$, let
$\fg_{\alpha}$ denote the corresponding root space,
and fix $Z_{\alpha} \in \fg_{\alpha}$.  

Let  $I = \{ s_{i_1}, \ldots, s_{i_k} \} \subset S$ be a set of simple reflections.  We will sometimes identify $I$ with the set $\{ i_1, \ldots, i_k \}$.
Let $W_I = W_{i_1, \ldots, i_k}$ denote the subgroup generated by $I$, and let $w_I$ denote the longest element of $W_I$.  
Let $W^I$ be the set of minimum length representatives of $W/W_I$. 
We write $P_I = P_{i_1, \ldots, i_k} =BW_IB$ for the standard parabolic subgroup of $G$ corresponding to $I$.  
If $I$ consists of a single element $s = s_{\ga}$, we write
$P_I = P_s = P_{\ga}$.  

\subsection{$K$-orbit closures for general $G$} \label{ss:K-orb}
The group $K$ acts on $X = G/B$ with finitely many orbits (see \cite{Wol74}, \cite{Mat79}). Write $V=K\backslash G/B$. If
$\mathcal{V}=\{g\in G\;|\;g^{-1}(\theta(g))\in N_G(H)\}$, then the inclusion $\mathcal{V}\hookrightarrow G$ induces an isomorphism $K\backslash\mathcal{V}/H\cong V =K\backslash G/B$.
Let $\dot{v} = g(v)$ denote a representative in $G$ or $\mathcal{V}$ for $v \in V$; by abuse of notation,
we sometimes use the same notation for an element of $V$ and a representative in $G$ or $\mathcal{V}$.
See \cite{Yam} for representatives of $V$ in $G$; \cite{Wys} is a good reference for the case
of $G = GL(n)$ and $K = GL(p) \times GL(q)$.

Let $G_v = \overline{K \dot{v} B}$, and $X_v=\overline{K\dot{v}B}/B$.  Define the length of $v$ to be $\ell(v) = \dim X_v$.
Define $\phi: V \to W$ by the rule $\phi(v) = g^{-1}(\theta(g)) \mbox{ mod } H$, where $g = g(v) \in \mathcal{V}$.

Because $H$ is $\theta$-stable, the involution $\theta$ acts on the set $\Phi$ of roots; if we view the root $\alpha$ as a character of $H$, then
$\theta(\alpha)(h) = \alpha(\theta(h))$.
Given $v \in V$, we define an involution $\theta_v$ of $\Phi$ by the formula $\theta_v\cdot\alpha = \phi(v) \left(\theta(\alpha)\right)$.  
Let $g = g(v)$. 
We say a simple root $\alpha \in \Delta$ has $v$-type (a) real if $\theta_v\cdot\alpha=-\alpha$, (b) compact imaginary if $\theta_v\cdot\alpha=\alpha$ and $\Ad(g)(Z_\alpha) \in\mathfrak{k}$, (c) noncompact imaginary if $\theta_v\cdot\alpha=\alpha$ and $\Ad(g)(Z_\alpha) \in\mathfrak{p}$, and (d) complex if $\theta_v\cdot\alpha\neq\pm\alpha$. 
We call $\alpha$ of type (d) 
a complex ascent if $\theta_v\cdot\alpha>0$, and a complex descent if $\theta_v\cdot\alpha<0$. (See \cite[Definition 2.5]{Vog3},\cite[\S2.4]{RS93}, or equivalently, 
\cite{Springer1985}.)  If $s = s_{\alpha}$ for $\alpha \in \Delta$, we say that the $v$-type of $s$ is the $v$-type of $\alpha$.  Note that if $\alpha$ is $v$-compact
imaginary, and $P_\alpha$ is the corresponding standard parabolic subgroup, then for every $\dot{v} \in G$, we have
\begin{equation} \label{e:compact-imag}
\dot{v} P_{\alpha}/B \subseteq K \dot{v} B/B
\end{equation}
by \cite[Lemma 5.1]{Vog3} (see, e.g., \cite[2.4]{RS93}).

There is a monoid product $*$ on $W$ given by
$$
w*s = \begin{cases}
ws & \mbox{ if } \ell(ws) > \ell(w) \\
w & \mbox{ if } \ell(ws) < \ell(w)
\end{cases}
$$
for $w \in W, s \in S$.
There is a right monoid action of $W$ on $V$: if $v \in V$ and $w \in W$, then $v*w = u$ if
$K \dot{u} B$ is the unique dense $K \times B$-orbit in $(K \dot{v} B)(B \dot{w} B)$
(see  \cite{RS93}).  If $I \subset S$, then $v*w_I$ corresponds to the open $K$-orbit in $K \dot{v} P_I/B$.
Given a $K$-orbit closure $\overline{K\dot{v}B}/B$ for $v\in V$, one can always find $v_0\in V$ and $w_1,\ldots,w_m\in W$ such that 
the orbit closure $X_{v_0}$ is smooth (in fact, such that $X_{v_0}$ is a closed orbit), and $v=v_0*w_1*\cdots*w_m$ (see \cite[Theorem 4.6]{RS90}).

The $\tau$-invariant of $w \in W$ or $v \in V$ is defined as
\[
\tau(w) =  \{s \in S \;|\;w*s=w\} \mbox{   and   }   \tau(v)=\{s \in S \;|\;v*s=v\}.
\]
for $w \in W$, $v \in V$.  Equivalently, 
$\tau(w)$ is the right descent set of $w$; in the language of \cite{Vog3}, $\tau(v)$ is the weak $\tau$-invariant of the trivial local system on the orbit $K \dot{v} B/B$.
Since
 $K\dot{v}B$ is the open $K \times B$-orbit
in $K \dot{v} P_{\tau(v)}$, we have
$\overline{K \dot{v} B} = \overline{K \dot{v} P_{\tau(v)}}$.  
Hence $\overline{K \dot{v} B}$ is right $P_{\tau(v)}$-invariant.

For $s \in S$, there is a proper surjective morphism
\begin{equation} \label{e:morphism}
\mu: G_v \times^B P_s/B \to X_{v \star s}.
\end{equation}
If $s$ has $v$-type real, compact imaginary, or complex descent, then $v \star s = v$ (that is, $s \in \tau(v)$).  If $s$ has
$v$-type noncompact imaginary or complex ascent, then $v \star s > v$; in this case, the map $\mu$ is generically finite.
If the $v$-type of $s$ is complex ascent, then $\mu$ is birational.
The $v$-noncompact imaginary simple reflections (or roots) are classified into two types.  For type I, the map $\mu$ is birational;
for type II, $\mu$ is generically 2-1.

Let $P = P_I$ and let $\rho: K\backslash G/B \to K\backslash G/P$ denote the 
natural surjection of finite sets.  
Let $V_I = \{ v \in V \mid v * w_P = v \}$.  For $v \in V_I$, let 
$\Gamma(v) = \{ w \in V \mid w * w_P = v \}$.
By \cite[Section 3.4]{RS93},
the map $V_I \to K \backslash G/P$, $v \mapsto K\dot{v}P/P$, is a bijection, so
$G = \sqcup_{v \in V_I} K \dot{v} P$.  Moreover, if $v \in V_I$, then
$$
\rho^{-1}(K \dot{v} P) = \sqcup_{w \in \Gamma(v)} K \dot{w} B.
$$
Given $x \in V$, we have $K \dot{x} P = K \dot{v} P$ for $v \in V_I$ if and only if $x \in \Gamma(v)$.
(Indeed, if $K \dot{x} P = K \dot{v} P$, then $K \dot{x} B \subseteq \rho^{-1}(K \dot{x} P) = \rho^{-1}(K \dot{x} B)$, so $x \in \Gamma(v)$.
Conversely, if $x \in \Gamma(v)$ then $\rho(K \dot{x} B) =  K \dot{v} P$.)  
This implies that $K \dot{x} P = K \dot{y} P$ if and only if $x * w_P = y * w_P$.

\subsection{$K$-orbit closures for $G = GL(n)$ and $K = GL(p) \times GL(q)$} \label{ss:K-orbit-particular}
Unless otherwise noted, for the remainder of Section \ref{s:background}, $G = GL(n)$, $\theta$ is the involution
given by conjugation by 
$
I_{p,q}=
\begin{bmatrix}
I_p & \\
 & -I_q
\end{bmatrix},
$
where $I_m$ indicates the $m\times m$ identity matrix, and $p+q = n$,
and $K = G^{\theta} \cong GL(p) \times GL(q)$
is the block diagonal subgroup of $G$ with blocks of sizes $p$ and $q$.  Let $B$ and $H = T$ denote the subgroups of upper triangular matrices and diagonal matrices in $G$, respectively.  
The simple roots are $\ga_i = \gre_i -  \gre_{i+1} $ for $i = 1, \ldots, n-1$.  Under the isomorphism $W \cong S_n$, 
$s_i$ is the transposition $(i, i+1)$.  For this $G$ and $K$, each $X_v$ is normal (see \cite{Brion2001}, Theorem 6 and the discussion following Corollary 2;
cf.~\cite{WWY}, which refers to \cite{Brion2001}.)  

This following (known) lemma implies that for this $G$ and $K$,
the $K$-orbits in $G/P$ are equivariantly simply connected, i.e., there are no nontrivial $K$-equivariant coverings of orbits.  
For lack of a precise reference we give the proof.

\begin{lemma} \label{l:equiv-sc}
Let $G = GL(n)$ and $K = GL(p) \times GL(q)$ as above.  The stabilizer in $K$ of any point in $G/P$ is connected.
\end{lemma}

\begin{proof}
The stabilizer in $K$ of $gP$ is $K \cap gPg^{-1}$.  Since $K = GL(p) \times GL(q)$, and a Levi factor of $P$ is a product of
general linear groups, the groups $K$ and $gPg^{-1}$ can each be viewed as open subsets of their Lie algebras $\fk$
and $\Ad(g) \fp$.  Therefore, $K \cap gPg^{-1}$ is an open subspace of the vector space $\fk \cap \Ad(g) \fp$, so it is connected.
\end{proof}

The set $V = K \backslash G/B$ can be parametrized by {\em clans}, which are defined as follows.
A clan $\gamma$ is a sequence of symbols $(a_1, \ldots, a_n)$, where each symbol is either a sign (that is $+$ or $-$) or an integer, such
that each integer occuring in $\gamma$ occurs exactly twice; and 
if $n(\gamma,+)$ (resp.~$n(\gamma,-)$) are the number of occurrences of the symbol $+$ (resp.~$-$) in the clan $\gamma$, then
$n(\gamma,+) - n(\gamma,-) = p-q$.  Clans $\gamma = (a_1, \ldots, a_n)$ and $\gamma' = (a_1', \ldots, a_n')$ are considered equal if
the following hold: $a_i$ is a sign if and only if $a'_i$ is a sign, and in that case $a_i = a'_i$; and  $a_i = a_j$ are equal integers if and only if
$a'_i = a'_j$ are equal integers.  For example, if $n = 7$, $p = 4$, $q = 3$, then
$(1,+,2,1,2,-,+)$ and $(2,+,1,2,1,-,+)$ are equal clans.  If $a_i$ is a natural number, the $a_j$ which is equal to $a_i$ is called the mate of $a_i$.  
We identify $v \in V$ with the corresponding clan,
and if we need to specify $p$ and $q$, we will write
$V_{p,q}$ for $V$.  
For brevity, we frequently omit commas, e.g.~we write $(1+1)$ instead of $(1,+,1)$.  The closed orbits correspond to clans whose entries are all signs.
The open $K$-orbit (i.e,.~the unique maximal orbit in the closure
ordering) corresponds to the clan
$$
v = \max(V_{p,q}) = \begin{cases}
(1\;2\;\cdots\;q\;+\;\cdots\;+\;q\;\cdots\;2\;1)  \mbox{  if  } p\geq q \\
(1\;2\;\cdots\;p\;-\;\cdots\;-\;p\;\cdots\;2\;1)   \mbox{ if } q\geq p.
\end{cases}
$$

Given clans $v^i \in V_{p_i, q_i}$, we can form a clan $v = (v^1, \ldots, v^r) \in V = V_{p,q}$ (where $p = \sum p_i$ and $q = \sum q_i $)
by concatenating the clans $v^i$.
We refer to $v^i$ as the $i$-th block of $v$; this depends on the particular
expression for $v$ as a concatenation.  We often write $v =(v^1 \mid  v^2 \mid \cdots \mid v^r)$ to emphasize the block structure.
 
The map $\phi: V \to W$ takes the clan $v = (a_1, \ldots, a_n)$ to the involution $\phi(v) \in W$ which switches $i$ and $j$ whenever $a_i$ and $a_j$ are equal integers,
and fixes $i$ if $a_i$ is a sign.  Since $\theta(\alpha) = \alpha$ for all $\alpha \in \Phi$,
 the action $\theta_v$ on $\Phi$
satisfies $\theta_v\cdot \alpha = \phi(v) (\alpha)$.  Hence, a simple root $s_i$ is imaginary if and only if $a_i$ and $a_{i+1}$ are signs.

The monoid action on clans is as follows (see \cite{MT}).  Let $v = (a_1, \ldots, a_n)$ be a clan.  
If $s_i \not\in \tau(v)$, then $v*s_i> v$.
This occurs if $\alpha_i$ is (1) a complex ascent, or 
(2) imaginary and noncompact.  In terms of the entries of $v$, this means:
\begin{enumerate}
\item \begin{enumerate}
\item $a_i$ and $a_{i+1}$ are unequal natural numbers, such that the mate of $a_i$ occurs before the mate of $a_{i+1}$;
\item $a_i$ is a sign, $a_{i+1}$ is a number, and the mate of $a_{i+1}$ occurs after position $i+1$;
\item $a_i$ is a number, $a_{i+1}$ is a sign, and the mate of $a_i$ occurs before position $i$;
\end{enumerate}
\item $a_i$ and $a_{i+1}$ are opposite signs.
\end{enumerate}
The clan $v * s_i$ differs from $v$ only in the $i$ and $i+1$ positions.  In case (1), $v*s_i$ is obtained by switching the $i$ and $i+1$ entries of $v$; in case (2), $v*s_i$ is obtained by replacing the $i$ and $i+1$ entries of $v$ with equal numbers.  
If both $a_i$ and $a_{i+1}$ are both signs, then as noted above, $s_i$ is imaginary.  In this case, if the signs are opposite then $s_i$ is noncompact imaginary, 
since $v*s_i > v$; if the signs are the same, then $s_i$ is compact imaginary, since $v*s_i = s_i$.  For this $G$ and $K$, all imaginary noncompact roots are of type I (cf.~\cite[Corollary 2]{Brion2001}).

\subsection{Smooth orbit closures} \label{ss:smooth}
The $K$-orbit closure $X_{v}$ is smooth if and only if 
$v = (v^1 \mid  \cdots \mid v^r)$, where each $v^j = \max(V_{p_j, q_j})$ (this is essentially stated in the proof of the main theorem of \cite{Mcg}).
Specifying such a $v$ is equivalent to specifying  decompositions $p = p_1 + p_2 + \cdots + p_r$ and $q = q_1 + q_2 + \cdots q_r$ where the
$p_i$ and $q_i$ are nonnegative integers and each $p_i+q_i>0$.

If $v =  (v^1 \mid  \cdots \mid v^r) = (a_1, \ldots, a_n)$, we say that $s_i$ straddles two blocks if $a_i$ is the last entry of a block and $a_{i+1}$
is the first entry of the next block.  If $X_v$ is smooth and each $v_j = \max(V_{p_j, q_j})$ then $S \setminus \tau(v)$ consists of
the $s_i$ which straddle two blocks such that at least one of which has more than one entry, or such that the blocks are a pair of opposite signs.
For example, if 
$v = (1+1| 2332| 4-4) = (v^1 \mid v^2 \mid  v^3)$, where $v^1= (1+1) = \max(V_{2,1})$, $v^2 = (2332) = \max(V_{2,2})$, and $v^3 =(4-4) =  \max(V_{1,2})$, then
$S \setminus \tau(v) = \{s_3, s_7 \}$.

\begin{lemma} \label{l:concatenation}
Suppose that $v \in V_{p,q}$ is of the form $v = (v^1 \mid \cdots \mid v^r)$, where each $v^k \in V_{p_k, q_k}$.  Suppose that $y \in V_{p,q}$ and $y \leq v$.
Then $y = (y^1 \mid \cdots \mid y^r)$, where each $y^i\leq v^i$ is in $V_{p_i, q_i}$. 
\end{lemma}

\begin{proof}
By \cite[Proposition 4.3.1]{Lar}, we have \(X_v=K\times^{P_K}(X_{v^1}\times\cdots\times X_{v^h})\), where \(P\) is the standard parabolic subgroup corresponding to \(S\smallsetminus\left\{n_1,n_1+n_2,\ldots,n-n_h\right\}\), and \(P_K=P\cap K\) is a parabolic subgroup of \(K\).
The \(K\)-stable subspaces of \(X_v\) correspond to the \(P_K\)-stable subspaces of \(X_{v^1}\times\cdots\times X_{v^h}\).
Since \(P_K\) contains a subgroup which acts as \(GL_{p_i}\times GL_{q_i}\) on \(X_{v^i}\), a closed \(K\)-stable subspace of \(X_v\) is equal to \(K\times^{P_K}(X_{y^1}\times\cdots\times X_{y^h})\) for some \(X_{y^i}\subseteq X_{v^i}\).
The claim follows.
\end{proof}

The following lemma will be used in computing tangent spaces of resolutions 
(see Theorem \ref{t:symmetric-mu} and Corollary \ref{c:T-weights}).

\begin{lemma} \label{l:tau-para}
If $X_{v_0}$ is smooth, then $G_{v_0} = \overline{K \dot{v}_0 B} = K \dot{v}_0 P_{\tau(v_0)}$.
Moreover, if $I_0$ is a set of simple reflections containing $\tau(v_0)$, then $G_{v_0'} = K \dot{v}_0 P_{I_0}$ is closed in
$G$ where $v_0'=v_0*w_{I_0}$.
\end{lemma}

\begin{proof}
We claim that $K\dot{v}_0P_{\tau(v_0)}/P_{\tau(v_0)}$ is closed in $G/P_{\tau(v_0)}$.  Indeed, if $K \dot{y} P_{\tau(v_0)}/P_{\tau(v_0)} \subseteq \overline{K\dot{v}_0P_{\tau(v_0)}}/P_{\tau(v_0)}$,
then $y * w_{\tau(v_0)} \leq v_0 * w_{\tau(v_0)} = v_0$.  Because $X_{v_0}$ is smooth, $v_0 = (v^1, \ldots, v^r)$ where $v^i = \max(V_{p_i, q_i})$.  Let $n_i = p_i+q_i$.
The element
 $w_{\tau(v_0)}$ is the longest
element in the subgroup $S_{n_1} \times \cdots \times S_{n_r}$ of $W$.  This implies that $y * w_{\tau(v_0)} = v_0$.  Hence $K \dot{y} P_{\tau(v_0)}/P_{\tau(v_0)} = K \dot{v}_0P_{\tau(v_0)}/P_{\tau(v_0)}$.
Thus, $K\dot{v}_0P_{\tau(v_0)}/P_{\tau(v_0)}$ is the only $K$-orbit in its closure, which implies the claim.

Since $K \dot{v}_0 B$ is the open $K \times B$ orbit in $K \dot{v}_0 P_{\tau(v_0)}$, we have $\overline{K \dot{v}_0 B} = \overline{K \dot{v}_0 P_{\tau(v_0)}}$.
The previous paragraph implies that this equals $K \dot{v}_0 P_{\tau(v_0)}$.  This proves the first assertion of the lemma. 

If $I_0 \supseteq \tau(v_0)$, there is a proper map $G/P_{\tau(v_0)} \to G/P_{I_0}$.  The image under this map of $K \dot{v}_0P_{\tau(v_0)}/P_{\tau(v_0)}$
is $K \dot{v}_0 P_{I_0} /P_{I_0}$, which is closed (since it is the image of a closed set).  Hence $K \dot{v}_0 P_{I_0} $ is closed in $G$.
\end{proof}

The next proposition plays an important role in the proof of Theorem \ref{t:fiberhomog}.

\begin{proposition} \label{p:orbitnumber}
 Let $G = GL(n)$ and $K = GL(p) \times GL(q)$.  
Suppose that $I \subset S$ is a set of commuting simple reflections. 
Suppose $n = n_1 + \cdots n_h$ and $n_i = p_i + q_i$ for
$i \in \{1, \ldots, h \}$.  Let 
$v = (v^1 \mid v^2 \mid \cdots \mid v^h) \in V_{p,q}$,
where $v^i \in V_{p_i, q_i}$.   Suppose that 
(a) $\ell(v*w_I)=\ell(v)+ | I |$,  and (b) each $s \in I$ straddles two blocks of $v$ (this follows from (a)
if each $v^i = \max(V_{p_i, q_i})$).  
Then the natural map $\rho: K \backslash \overline{KvB}/B \rightarrow K \backslash \overline{KvP_I}/P_I$
is a bijection.
\end{proposition}
\begin{proof}
Since $\pi_v:  \overline{KvB}/B \rightarrow   \overline{KvP_I}/P_I =  \overline{Kv * w_I P_I}/P_I $ is surjective, the map $\rho$
is surjective.  Hence, given any $w \in V_I$ with
$w \leq v * w_I$, there exists $x \in V$, $x \leq v$, such that $\pi_v(K x B/B) \subseteq K w P_I/P_I$.  
We must show that this $x$ is unique.  Observe that  $K x B/B \subseteq \pi_v^{-1}(K w P_I/P_I)$
is equivalent to the conditions $x \in \Gamma(w)$ and $x \leq v$.
It suffices to show that
the entries of $x$ are determined by these conditions.  

Because $x \leq v$, we can write
$x=(x^1 \mid \cdots \mid x^h)$.
The entries of $x$ coincide with those of
$w = x*w_I$ except in positions $m, m+1$, for $s_m \in I$. 
Suppose that $s_m \in I$.  Since $I$ consists of commuting simple reflections,
any other element $s_{m'}$ of $I$ satisfies $|m - m'| \geq 2$.  This implies that the entries $(w_m, w_{m+1})$ are determined by the entries
$(x_m, x_{m+1})$.  The description of the monoid action
implies that one of the following holds, where $\epsilon$ denotes a sign $+$ or $-$, and $a, b$ denote integers.
\begin{enumerate}
\item $(x_m, x_{m+1})$ and $(w_m, w_{m+1})$ are pairs of equal signs, and   $(x_m, x_{m+1}) =(w_m, w_{m+1})$. 
\item $(x_m, x_{m+1})$ are opposite signs, and $(w_m, w_{m+1})$ are equal integers.
\item $(x_m, x_{m+1}) = (a, \epsilon)$, where the mate of $a$ occurs before position $m$, and
$(w_m, w_{m+1}) = (\epsilon, a)$.
\item $(x_m, x_{m+1}) = (\epsilon, a)$, where the mate of $a$ occurs after position $m+1$, and $(w_m, w_{m+1}) = (a, \epsilon)$, 
\item $(x_m, x_{m+1}) = (a, b)$, where the mate of $a$ occurs before position $m$, the mate of $b$ occurs after position $m+1$,
and $(w_m, w_{m+1}) = (b,a)$.
\end{enumerate}
 From this
list, we can see that $(w_m, w_{m+1})$ uniquely determines $(x_m, x_{m+1})$.  Indeed, the only possible ambiguity would
be if $(w_m, w_{m+1})$ are equal integers; then we are in case (2), so $(x_m, x_m+1)$ is either
$(+,-)$ or $(-,+)$.  The number of $+$ signs minus the number of $-$ signs in the $k$-th block of $x$
is $p_k - q_k$.  Therefore, whether $x_m$ is $+$ or $-$ is determined by the other entries
of $x$ in the same block, and similarly for $x_{m+1}$.  Hence there is only one possibility for $(x_m, x_{m+1})$.
We conclude that $x$ is the unique element in $V$ with $x \in \Gamma(w)$ and $x \leq v$.  
\end{proof}

\begin{remark} \label{r:birational}
Let $v \in V$ and $I \subseteq S$.  The restriction of $\pi: G/B \to G/P_I$ yields a proper surjective morphism $\pi_v: X_v = \overline{K\dot{v}B}/B \to \overline{K\dot{v}P_I}/P_I$.
If $\ell(v * w_I) = \ell(v) + \ell(w_I)$, then the map $\pi_v$ is birational (since $G = GL(n)$ and $K = GL(p) \times GL(q)$; for
general $G$ and $K$, $\pi_v$ is generically finite).
Indeed, the hypothesis $\ell(v * w_I) = \ell(v) + \ell(w_I)$ is equivalent to the statement that $\dim X_{v*w_I} = \dim X_v + \dim P/B$.  The map
$$
\overline{K \dot{v} P_I}/B = \overline{K v* w_I B}/ B \to \overline{K \dot{v} P_I}/P_I
$$
is a fiber bundle with fibers isomorphic to $P_I/B$, so $\dim X_{v*w_I} = \dim  \overline{K \dot{v} P_I}/P_I + \dim G/P$.  
Hence $\dim X_v = \dim  \overline{K \dot{v} P_I}/P_I$, so $\pi_v$ is generically finite.  In particular it is finite over the open orbit in $\overline{K v P_I}/P_I$.
Since the orbits are equivariantly simply connected, $\pi_v$ is an isomorphism over this orbit, so it is birational.   
The map
$\rho: K \backslash X_v \rightarrow K \backslash \overline{KvP_I}/P_I$ takes $K \dot{x} B/B$ to $\pi_v(K \dot{x} B/B) = K \dot{x} P_I/P_I$.
The conclusion of Proposition \ref{p:orbitnumber} that $\rho$ is a bijection means that  $\pi_v^{-1} (K \dot{x} P_I/P_I)$ consists
of a single $K$-orbit for $x \leq v$.
\end{remark}

\begin{example} \label{e:orbitnumber}
Let $G = GL(4) \supset K = GL(2) \times GL(2)$.  The closure order for $K$-orbits on $G/B$ for this case is depicted in \cite[Fig.~2]{Wys}.
Let $v^1 = v^2 = \max(V_{1,1})$, and let $v =(v^1 \mid v^2) = (1 1 | 2 2)$.
Let $I=\{s_2\}$.  Under the map $\pi: G/B \to G/P_I$, orbit closures $X_u$ and $X_{u'}$ have the same image if and only
if $u*s_2 = u'*s_2$.  Thus, the preimage under $\pi$ of a $K$-orbit consists of either $1$ or $2$ $K$-orbits.  The map $\pi_v$ takes $X_u$ (for $u \leq v$) to the orbit indexed
by $u*s_2$.  In particular, $X_v$ maps to the orbit indexed by $v*s_2 = (1 2 1 2)$.  
The inverse image $\pi_v^{-1}(G/P_I)$ is the union of the $K$-orbits contained in $X_v$.  From \cite[Fig.~2]{Wys}, we can verify that the
inverse image under $\pi_v$ of any $K$-orbit is a single $K$-orbit, confirming that $\rho$ is a bijection.  The picture below describes what the map $\pi_v$ does to orbits; the left diagram
records the clan $u = (u^1 \mid u^2)$ corresponding to the orbit closure $K u B \subset X_v$, and the right diagram records the element $u* s_2 \in V_I$
describing the orbit $\pi_v(KuB)$.

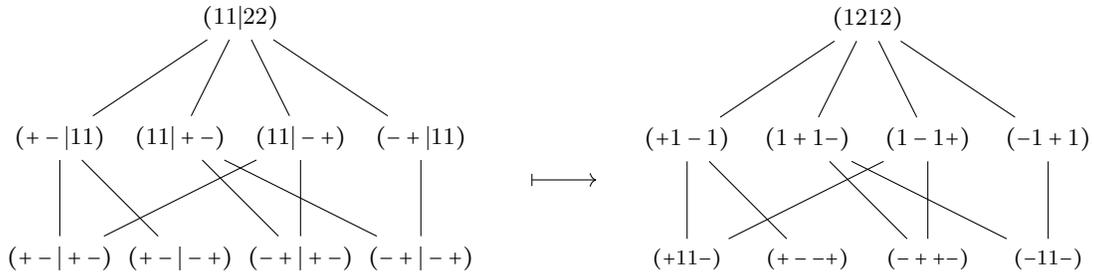
\begin{figure}[H]
\begin{tikzpicture}[scale=.4]

\node (d) at (-6,0) {\footnotesize{$(+-|+-)$}};
\node (e) at (-2,0) {\footnotesize{$(+-|-+)$}};
\node (g) at (6,0) {\footnotesize{$(-+|-+)$}};
\node (h) at (-6,4) {\footnotesize{$(+-|11)$}};
\node (i) at (-2,4) {\footnotesize{$(11|+-)$}};
\node (j) at (2,4) {\footnotesize{$(11|-+)$}};
\node (k) at (6,4) {\footnotesize{$(-+|11)$}};
\node (v) at (0,8) {\footnotesize{$(11|22)$}};
\node (l) at (2,0) {\footnotesize{$(-+|+-)$}};
\draw  (l) -- (j);
\draw  (l) -- (i);

\draw  (d) -- (h);

\draw  (e) -- (h);
\draw  (d) -- (j);
\draw  (g) -- (i);
\draw  (g) -- (k);
\draw  (h) -- (v);

\draw  (i) -- (v);

\draw  (j) -- (v);

\draw  (k) -- (v);
\end{tikzpicture}
\begin{tikzpicture}[scale=.4] 
\draw [black, |->, 
    shorten >=-20pt
] (4,3) -- (4.4,3) ;
\node (d) at (4,0) {\quad\quad\quad};
\end{tikzpicture}
\quad\begin{tikzpicture}[scale=.4] 

\node (d) at (-6,0) {\scriptsize{$(+11-)$}};
\node (e) at (-2,0) {\scriptsize{$(+--+)$}};
\node (g) at (6,0) {\scriptsize{$(-11-)$}};
\node (h) at (-6,4) {\footnotesize{$(+1-1)$}};
\node (i) at (-2,4) {\footnotesize{$(1+1-)$}};
\node (j) at (2,4) {\footnotesize{$(1-1+)$}};
\node (k) at (6,4) {\footnotesize{$(-1+1)$}};
\node (v) at (0,8) {\footnotesize{$(1212)$}};
\node (l) at (2,0) {\scriptsize{$(-++-)$}};
\draw  (l) -- (j);
\draw  (l) -- (i);

\draw  (d) -- (h);

\draw  (e) -- (h);
\draw  (d) -- (j);
\draw  (g) -- (i);
\draw  (g) -- (k);
\draw  (h) -- (v);

\draw  (i) -- (v);

\draw  (j) -- (v);

\draw  (k) -- (v);
\end{tikzpicture}
\captionof{figure}{Images of clans $u\leq v_0$ under $\pi_v$}
\label{Fig1}
\end{figure}
\end{example}

\subsection{Some results on the $\tau$-invariant} \label{ss:tau}
For later use, we record the following lemmas.  We continue assuming $G = GL(n)$ and $K = GL(p) \times GL(q)$.

\begin{lemma} \label{l:tau1} Let $v = (a_1, \ldots, a_n) \in V_{p,q}$.  Suppose $s = s_{\epsilon_i - \epsilon_{i+1}}$ and $s' = s_{\epsilon_j - \epsilon_{j+1}}$ are unequal commuting 
simple reflections, neither of which is in $\tau(v)$.  Suppose in addition that 
if $(a_i, a_{i+1})$ and $(a_j, a_{j+1})$ are both ordered pairs of natural numbers, then
$(a_i, a_{i+1}) \neq (a_j, a_{j+1})$ as ordered pairs.
Then $s' \not\in \tau(v*s)$, and $v < v*s < v* s *s'$.  
\end{lemma}

\begin{proof} The monoid action, in terms of clans, is described explicitly in Section \ref{ss:K-orbit-particular}.  The lemma follows by
examining what happens for each of the possibilities for $(a_i, a_{i+1})$ and $(a_j, a_{j+1})$ described there.
\end{proof}

Assuming the other hypotheses of the lemma, the condition that $(a_i, a_{i+1})$ and $(a_j, a_{j+1})$ are distinct ordered pairs of natural numbers is automatically
satisfied if $v$ is a concatenation of smaller clans, and $s$ and $s'$ both straddle blocks of $v$.

\begin{example}
Even if the other hypotheses of the lemma hold,  the conclusion of the lemma can fail
 if $(a_i, a_{i+1}) = (a_j, a_{j+1})$ are 
equal ordered pairs of natural numbers 
 (e.g.~for $v = (1212),  I = \{ s, s' \}$ where $s = s_{\epsilon_1 - \epsilon_2}$ and $s' = s_{\epsilon_3 - \epsilon_4}$.
In this example, the inequality $\ell(v*w_I) \leq \ell(v) + |I|$ is strict.  The next
lemma gives some conditions which imply equality.
\end{example}

\begin{lemma} \label{l:tau2} Let $v \in V$ be a concatenation of clans.  Suppose $I = \{ s_{i_1}, \ldots, s_{i_k} \}$ is a set
of commuting simple reflections, each of which straddles two blocks of $v$.  The following are equivalent:
\begin{enumerate}
\item $I \cap \tau(v) = \emptyset$.
\item $v < v*s_{i_1} < \cdots < v*s_{i_1}* \cdots * v* s_{i_k} = v* w_I$.
\item $\ell(v* w_I) = \ell(v) + |I|$.
\end{enumerate}
\end{lemma}

\begin{proof}
$(1) \Rightarrow (2)$: Suppose $I \cap \tau(v) = \emptyset$.  Applying Lemma \ref{l:tau1} repeatedly yields (2).

$(2) \Rightarrow (3)$: If $v < v*s$, then $\ell(v*s) = \ell(v) + 1$.  Hence the chain of inequalities in (2) implies (3).

$(3) \Rightarrow (1)$: We prove the contrapositive.  Suppose $I \cap \tau(v) \neq \emptyset$, so  $s_{i_r} \in \tau(v)$ for some $r$.
Let $I' = I \smallsetminus \{ s_{i_r} \}$.  Then $v* w_I = v*w_{I'}$, so $\ell(v* w_I) = \ell(v*w_{I'}) \leq \ell(v) + |I| - 1$, so $(3)$ does not hold.
\end{proof}

Note that $(1) \Leftarrow (2) \Leftrightarrow (3)$ for general $G$ and $K$.  Only the implication $(1) \Rightarrow (2)$ uses the hypothesis that
$G = GL(n)$ and $K = GL(p) \times GL(q)$.  We do not know if this hypothesis is necessary.

\begin{lemma} \label{l:compact-im}
Suppose $X_{v_0}$ is smooth, and let $y \leq v_0$.   If $s$ is a simple reflection such that $s \in \tau(y)$ and $s \not\in \tau(v_0)$, then
$s$ is $y$-compact imaginary.
\end{lemma}
\begin{proof}
Since $X_{v_0}$ is smooth, $v_0=(v^1 \mid \cdots \mid v^h)$ for some $h$, where $v^i = \max(V_{p_i, q_i})$. 
By Lemma \ref{l:concatenation},
$y = (y^1 \mid \cdots \mid y^h)$, where $y_i \in V_{p_i, q_i}$.  

Now suppose $s \in \tau(y)$ and $s \not\in \tau(v_0)$.  Because $s \not\in \tau(v_0)$, $s$ must straddle the $k$ and $k+1$ blocks of $v_0$.  If $y^k$ ends with a number, then the mate of this number occurs in $y^k$ (as $y^k \in V_{p_k, q_k}$) -- in particular, it occurs prior to the last entry of $y^k$.  Similarly, if $y^{k+1}$ begins with a number, then the mate of this number occurs in $y^{k+1}$, so it occurs after the first entry of $y^{k+1}$.  In either case, $y * s > y$, so $s \not\in \tau(y)$, contradicting our assumption.  Hence $y^k$ ends with a sign and $y^{k+1}$ begins with a sign.  Our discussion of the monoid action implies that the signs must be the same, so $s$ is $y$-compact imaginary.
\end{proof}

\subsection{$T$-fixed points of smooth orbit closures} \label{ss:T-fixed}
In this section we describe the $T$-fixed points of smooth $K$-orbit closures in $G/B$, for
$G = GL(n) \supset K = GL(p) \times GL(q)$.  We will discuss $T$-fixed points further when we discuss resolutions.

\begin{lemma}\label{lem6.3}
All $T$-fixed points in $X = G/B$ are contained in closed $K$-orbits.
\end{lemma}
\begin{proof}
Identify $X$ with the set of Borel subgroups of $G$.
The $T$-fixed points in $G/B$ are the Borel subgroups containing $T$.   For $G = GL(n)$ and $K = GL(p) \times GL(q)$, 
the involution is $\theta=\Ad(I_{p,q})$.  Note that  
$\theta(t)=t$ for all $t\in T$; thus, $\theta(\alpha)(t)=\alpha(t)$.  Hence $\theta$ acts trivially on the set of roots. So, any Borel subgroup $B_1$ 
containing $T$ is $\theta$-stable, i.e., $\theta(B_1)=B_1$.   
According to \cite[Lemma 5.8]{Milicic}, the union of closed orbits is precisely the set of $\theta$-stable Borel subgroups.  Hence any $B_1 \in X^T$
is contained in a closed orbit.
\end{proof}

Note that the statement of Lemma \ref{lem6.3} is valid for general $G$ and $K$.

\begin{example}
Let $v=(1212) \in V_{2,2}$. Then, as is apparent in \cite[Figure 2]{Wys}, each of the $6$ closed $K$-orbits in $X = G/B$ is contained in $X_v$.
Hence, $X_v^T = X^T$ is in bijection with the Weyl group $W$.  
\end{example}

Given a clan $u$ consisting only of signs, let $w_u$ denote the permutation such that in $1$-line notation, the entries $1, \ldots, p$ (resp.~ $p+1, \ldots, n$) appear in increasing order
in the positions where the entries of $u$ are $+$ (resp.~$-$).

\begin{proposition} \label{p:fixedpoint-closed}
If $X_u$ is a closed $K$-orbit in $G/B$, then
$$
 X_u^T = \{ w B/B \mid w \in W_K w_u \}.
$$ 
In other words, if $u = (a_1, \ldots, a_n)$, then 
$wB/B$ is in $X_u^T$ if and only 
$w_i \in [p]$ $\Leftrightarrow$ $a_i = +$, and $w_i \in [n] \smallsetminus [p]$
$\Leftrightarrow$ $a_i = -$.  
\end{proposition}

\begin{proof}
The procedure given in \cite{Yam} (described also in \cite{Wys}) for constructing an element of $X_u$ constructs exactly the elements
$wB$ described in the second statement of the proposition.  The first statement then follows from the fact that $W_K = S_p \times S_q \subset W$.
\end{proof}

\begin{remark}
If we start from the fact that $w_u B/B \in X_u$, then the equality $X_u^T = W_K \cdot w_uB/B$  follows because
the stabilizer
in $K$ of $w_u B$ is equal to $B_K = B \cap K$, so  $X_u \cong K/B_K$, the flag variety for $K$.
\end{remark}

\begin{example}
Let $u=(+ + - + - -+)\in V_{4,3}$. Then $w_u = w=(1253674)\in W$, and
$$
W_K w  = \{(i_1,\ldots,i_7)\:|\;\{i_1,i_2,i_4,i_7\}=\{1,2,3,4\},\{i_3,i_5,i_6\}=\{5,6,7\}\}.
$$
\end{example}

 Let $v_0 = (v^1 \mid \cdots \mid v^r)$ with $v^i = \max(V_{p_i, q_i})$.  
Given $w \in W$, define the $i$-th
 block of $w$ to be the sequence of entries of $w$ in the positions corresponding to the $i$-th block $v^i$ of $v_0$.
Let $W(v_0)$ denote the set of $w \in W$ such that for each $i$,
the $i$-th block of $w$ contains $p_i$ entries in $[p]$ and $q_i$ entries in $[n] \smallsetminus [p]$.
The set $W(v_0) \cap W^{\tau(v_0)}$ is a set of right $W_{\tau(v_0)}$-coset representatives in $W(v_0)$.

Since $W_{\tau(v_0)}$ consists of the elements of $W$ which preserve the blocks under right
multiplication,
$W(v_0) \cap W^{\tau(v_0)}$ consists of the elements in $W(v_0)$ such that the entries in each block
are in increasing order.  If $X_u$ is a closed orbit contained in $X_{v_0}$, then 
\begin{equation} \label{e:wv0}
W(v_0) = W_K w_u W_{\tau(v_0)}.
\end{equation}

\begin{proposition} \label{p:fixedpoint-smooth}
If $X_{v_0}$ is smooth, then 
\begin{equation}\label{eqn4.4}
X_{v_0}^T = \{ w B/B \mid w \in W(v_0) \}. 
\end{equation}
Hence $(K v_0 P_{\tau(v_0)}/P_{\tau(v_0)})^T$ is in bijection with $w P_{\tau(v_0)}/P_{\tau(v_0)}$ for $w \in W(v_0) \cap W^{\tau(v_0)}$.
\end{proposition}

\begin{proof}
We have 
\begin{equation}
X_{v_0}^T = \bigcup_{u \leq v_0, X_u \text{closed}}X_u^{T} =  \{ w B \mid w \in W(v_0) \},
\end{equation}
where the first equality follows from Lemma \ref{lem6.3}, and the second from Proposition \ref{p:fixedpoint-closed}.
This proves \eqref{eqn4.4}.  The second assertion follows because $W_{\tau(v_0)}$
acts freely and transitively on the fibers of $(G/B)^T \to (G/P_{\tau(v_0)})^T$.
\end{proof}

\begin{remark} \label{r:fixedpoint-smooth}
An alternative argument for \eqref{eqn4.4} is as follows.  By Lemma \ref{l:tau-para}, $X_{v_0} = K \dot{v}_0 P_{\tau(v_0)}/B$.  Thus, $w_u B/B \in K \dot{v}_0 P_{\tau(v_0)}/B$,
so $X_{v_0} = K w_u P_{\tau(v_0)}/B$.
Multiplication in $G$ induces a surjection
\begin{equation} \label{e:fixedpoint-smooth}
f: K w_u B \times^B P_{\tau(v_0)}/B \to K w_u P_{\tau(v_0)}/B = X_{v_0}.
\end{equation}
Applying the description of $X_u^T$ above, together with the description of fixed points in mixed spaces (see \cite[Section 3]{GZ}, cf.~Section \ref{ss:mixed}), we see that
$$
(K w_u B \times^B P_{\tau(v_0)}/B)^T = \{ [w_1, w_2 B/B] \mid w_1 \in W_K w_u, w_2 \in W_{\tau(v_0)}\}.
$$
Thus,
$$
X_{v_0}^T = f((K w_u B \times^B P_{\tau(v_0)}/B)^T) = \{ w B \mid w \in W_K w_u W_{\tau(v_0)} \}.
$$
Equation \eqref{eqn4.4} follows since $W_K w_u W_{\tau(v_0)} = W(v_0)$, as noted above.
Note that \eqref{e:fixedpoint-smooth} need not induce a bijection on $T$-fixed points.  For example, if $v_0 = (1 1 | 2 + 2)$, $u = (+ - - + +)$, then
$w_u = (1 4 5 2 3)$.  The points $[s_2 w_u, eB/B]$ and $[w_u, s_4 B/B]$ are different but have the same image under $f$.
\end{remark}

\section{Resolutions of $K$-orbit closures} \label{s:resolution}
In this section we review some results about resolutions of $K$-orbit closures for general
$G$ and $K$.  We work more generally than the particular family of resolutions
which is our focus, because some of our results (for example, descriptions of
tangent spaces) apply in this wider setting.

\subsection{Bott-Samelson type resolutions of $K$-orbit closures} \label{ss.generalities-resolutions}
Let $v_0 \in V$, and let $w = s_{i_1} s_{i_2} \cdots s_{i_k}$ be a reduced expression.  For each $j = 1, \ldots, k$, let
$v_j = v_{j-1} * s_{i_j}$.  Let $v = v_k$.  
There is a natural map
\begin{equation} \label{e:generalmu}
\mu: G_{v_0} \times^B P_{i_1} \times^B \cdots \times^B P_{i_k}/B \to X_v.
\end{equation}
By induction, reasoning as in the discussion of \eqref{e:morphism},
we see the following.
Suppose that $\ell(v) = \ell(v_0) + \ell(w)$. Then this map is generically finite.
Each $s_{i_j}$ is of $v_{j-1}$-type either (1) complex ascent, or (2) non-compact imaginary.
If in addition whenever $s_{i_j}$ is of $v_{j-1}$-type noncompact imaginary it is of type $I$ (which is always the case
for $G = GL(n)$ and $K = GL(p) \times GL(q)$), then $\mu$ is birational, so
if $X_{v_0}$ is smooth, then $\mu$ is a resolution of singularities.

Suppose $s_{i_1}, \ldots, s_{i_k}$ are all distinct, and write $I = \{ s_{i_1}, \ldots, s_{i_k}  \}$. 
Recall that $Y_w$ denotes the Schubert variety $\overline{BwB}/B$. 
By \cite[p. 50]{Springer98}, $Y_w=\overline{BwB}/B=P_{i_1}P_{i_2}\cdots P_{i_k}/B=P_I/B$.
We claim that the multiplication map yields an isomorphism
\begin{equation} \label{e.Bott-Samiso}
m:P_{s_{i_1}} \times^B P_{s_{i_2}} \times^B \cdots  \times^B P_{s_{i_k}}/B \to Y_w = P_I/B.
\end{equation}
Indeed, let $I' = I \smallsetminus \{ i_1 \}$.  We may assume by induction that
$P_{i_2} \times^B \cdots \times^B P_{i_r}/B \to P_{I'}/B$ is an isomorphism.
 We must show that
$m: P_{i_1} \times^B P_{I'}/B \to P_I/B$
 is an isomorphism. 
To see this, note that $P_{i_1} \cap P_{I'} = B$, so the fiber of $m$ over $eB$ is a point.  As the point $eB$ is in the closure of every $B$-orbit on $P_I/B$, semicontinuity of dimension 
implies that all fibers of $m$ are $0$-dimensional.
By Zariski's connectedness theorem, the fibers of $m$ are connected, and thus $m$ is bijective. Since $Y_w$ is normal,  Zariski's Main Theorem implies that $m$ is an isomorphism,
proving the claim.

 From \eqref{e:generalmu} and the isomorphism \eqref{e.Bott-Samiso}, we obtain
\begin{equation} \label{e:morphism3}
\mu: G_{v_0} \times^B P_I/B \to X_v,
\end{equation}
which is generically finite if $\ell(v) = \ell(v_0) + \ell(w)$, and birational under the additional hypothesis above.

If $\alpha_i$ is a simple root,
let $L_i = L_{\alpha_i}$ denote the standard Levi subgroup of $P_i = P_{\alpha_i}$, and let
$L'_i$ denote its derived subgroup, which is isomorphic to $SL(2)$ or $PGL(2)$.

\begin{lemma}\label{l:commutingsimple}
Let $I = \{s_{i_1} , . . . , s_{i_r}\} \subset S$ be a set of commuting simple reflections. 
Then
$$
P_I/B \cong P_{i_1}/B \times P_{i_2}/B \times \cdots \times P_{i_r}/B \cong(\mathbb{P}^1)^r.
$$
The group $L = L_I$
acts transitively on $P_I/B$.
\end{lemma}
\begin{proof}
Let $P_I = LU$ be a Levi decomposition, and let $L'$ denote the closed subgroup of $L$ corresponding to the semisimple
part $\fl'$ of $\fl$.  Since $I$ is a set of commuting simple reflections, $L' \cong L'_{i_1} \times \cdots \times L'_{i_r}$.  Then
$$
P_I /B \cong L'/(L' \cap B) \cong L'_{i_1}/(L'_{i_1} \cap B) \times \cdots \times L'_{i_r}/(L'_{i_r} \cap B) \cong P_{i_1}/B \times P_{i_2}/B \times \cdots \times P_{i_r}/B.
$$
Since $P_{i_k}$ is a minimal parabolic subgroup, each $P_{i_k}/B$ is isomorphic to $\mathbb{P}^1$.  
\end{proof}

\subsection{More general constructions of resolutions} \label{ss:moregeneral}
The construction in \eqref{e:generalmu} can be generalized.  
Let \(v_0\in V\) and let \(w_1,\ldots,w_m\in W\).  To simplify notation, we will sometimes
write $u_0 = v_0$ and $u_k = w_k$ for $1 \leq k \leq m$.  Note that
$u_0$ will correspond to a $K$-orbit, and $u_k$ to a $B$-orbit for $k \geq 1$. 

For \(1 \leq k\leq m\), let
\(J_k=\tau(u_{k-1})\cap\tau(u_{k}^{-1})\).
Suppose \(J_{m+1}\subseteq\tau(w_{m+1})\).
Let
$G_{v_0} = \overline{K v_0 B}$, and
for every $1\leq k\leq m$,  let
$G_{w_k} = \overline{Bw_kB}$.
Set 
\begin{equation}\label{e:generalZ}
Z=G_{v_0}\times^{P_{J_1}} G_{w_1}\times^{P_{J_2}}\cdots\times^{P_{J_m}}G_{w_m}/P_{J_{m+1}}.
\end{equation}
Multiplication gives $\mu: Z \to X_v = \overline{K \dot{v }B}/P_{J_{m+1}}$, where
$v = v_0*w_1 * \cdots * w_m$.  If $Z$ is smooth and $\mu$ is birational, then
$\mu$ is a resolution of singularities.  

If $G$ is simply laced, $w \in W$, and $G_w = \overline{BwB}$ is smooth, then 
by \cite[Theorem 5.5]{Lar2024}, $G_w$ is isomorphic to
an iterated fiber bundle.  Precisely, there are subsets $I_1, \ldots, I_k$ of $S$,
with $I_1 = \tau(w^{-1})$, $I_k = \tau(w)$, such that the multiplication map
\begin{equation} \label{e:smoothSchubert}
P_{I_1} \times^{R_1} P_{I_2} \times^{R_2} \times \cdots \times^{R_{k-1}} P_{I_k} \stackrel{\cong}{\longrightarrow} G_w,
\end{equation}
is an isomorphism, where
$R_j = P_{I_j} \cap P_{I_{j+1}}$.  If $Z$ is smooth, then so is each $G_{w_i}/B$.  
By replacing each $G_{w_i}$ by an iterated fiber
bundle of the form \eqref{e:smoothSchubert}, and reindexing (changing $m$ if necessary), we obtain
\begin{equation}\label{e:generalZ-A}
Z= G_{v_0} \times^{P_{J_1}} P_{I_1} \times^{P_{J_2}} \cdots \times^{P_{J_m}} P_{I_m}/P_{J_{m+1}}.
\end{equation}
Note that $P_{I} = G_{w_I}$, where $w_I$ is the longest element of the subgroup of $W$ generated by $I$.
Thus, \eqref{e:generalZ-A} is the special case of \eqref{e:generalZ} where $G_{w_k} = P_{I_k}$.  
If $Z$ is smooth and $G$ is simply laced, then $Z$ can always be expressed
as in \eqref{e:generalZ-A}.  
If $G_{v_0} = K \dot{v}_0 P_{\tau(v_0)}$, then 
\eqref{e:generalZ-A} becomes
\begin{equation}\label{e:generalZ-A2}
Z=  K \dot{v}_0 P_{\tau(v_0)} \times^{P_{J_1}} P_{I_1} \times^{P_{J_2}} \cdots \times^{P_{J_m}} P_{I_m}/P_{J_{m+1}}.
\end{equation}
By Lemma \ref{l:tau-para}, $G_{v_0} = K \dot{v}_0 P_{\tau(v_0)}$ if $G = GL(n)$ and $K = GL(p) \times GL(q)$
and $G_{v_0}$ is smooth, so under these hypotheses, we can assume that $Z$ is as in \eqref{e:generalZ-A2}.

By substituting $G_{v_0} = G_{v_0}  \times^{P_{\tau(v_0)}} P_{\tau(v_0)}$ into \eqref{e:generalZ-A} 
or \eqref{e:generalZ-A2}, we can assume $J_1 = I_1 = \tau(v_0)$ if desired.  This assumption implies
that in \eqref{e:generalZ-A2}, all of the successive quotients $K v_0 P_{\tau(v_0)} /P_{J_1}, P_{I_1} /P_{J_2}, P_{I_2}/P_{J_3}, \ldots$ are homogeneous varieties.
This will facilitate the computation of tangent spaces in Section \ref{ss:tangentK}.

\section{Irreducible characteristic cycles} \label{s:irreducible}
This section contains our main results, which concern the following family of resolutions for $G = GL(n)$ and $K = GL(p) \times GL(q)$.
Let $X_{v_0}$ be a smooth orbit closure in $G/B$, $I$ a set of commuting simple reflections, and
such that  $\ell(v_0* w_I) = \ell(v_0)+ |I|$.  Let $v = v_0 * w_I$, and write $Z_v = Z_{v_0, I} = G_{v_0} \times^B P_I/B$.
Then
\begin{equation} \label{e:morphism2}
\mu: Z_v = G_{v_0} \times^B P_I/B \to X_v
\end{equation}
is a resolution of singularities.  In Section \ref{ss:small} we recall a result of Larson characterizing when the morphism
\eqref{e:morphism2} is small.  In Section \ref{ss:fiber-cc} we prove that this morphism has smooth and strongly reduced fibers.
As a consequence, when the resolution is small, the characteristic cycle of $X_v$ is irreducible.  
In Section \ref{ss:isomorphic}, we give some alternative presentations of the resolution \eqref{e:morphism2}, which are
useful in analyzing the fibers.

\subsection{Smallness} \label{ss:small}
The following result of Larson (see \cite[Prop.~5.3.2]{Lar}) characterizes when \eqref{e:morphism2} is a small resolution.

\begin{proposition} \label{p:small}
Let $G = GL(n)$ and $K = GL(p) \times GL(q)$.  Let $X_{v_0}$ be a smooth orbit closure in $G/B$, and let $I$ be a set of commuting simple reflections such that  $\ell(v_0* w_I) = \ell(v_0)+ |I|$ and $\ell(y* w_I) = \ell(y) + |I|$ for all $y \leq v_0$ with $\ell(y) = \ell(v_0) - 1$.   Then $\mu: G_{v_0} \times^B P_I/B \to X_v$
is a small resolution.
\end{proposition}

The conditions of the proposition are made  explicit in terms of the clan $v_0$ in \cite[pp.~65-66]{Lar}.  The results can be summarized as follows.
Since $v_0$ is smooth, we have $v_0 = (v^1, \cdots, v^h)$, where $v^k = \max(V_{p_k,q_k})$.  Then $v_0$ satisfies the criteria of the proposition
if and only if for all $s_{i} \in I$, the clan $v_0$ does not contain one of the following four configurations:
\begin{equation} \label{e:badconfig}
\epsilon | a a, \hspace{.2in} a a | \epsilon,  \hspace{.2in} \epsilon | a \epsilon,  \hspace{.2in} \epsilon a | \epsilon.
\end{equation}
Here, the bar is assumed to be between entries $i$ and $i+1$ of $v_0$, where $i$ and $i+1$ are indices corresponding respectively
to the end of a block and the beginning of a block;
$\epsilon$ denotes a sign (either $+$ or $-$), and $a$ is a natural number.
See \cite[p.65, (5.10)]{Lar} for more details. 

\subsection{Fibers and characteristic cycles} \label{ss:fiber-cc}
In this section we prove (Theorem \ref{t:fiberhomog}) that the resolution $\mu$ of \eqref{e:morphism2} has smooth and strongly reduced fibers.  With an additional hypothesis, this resolution is also small, so the characteristic cycle of the target is irreducible (Theorem \ref{t:irreducible}).  

\begin{theorem} \label{t:fiberhomog}
Let $G = GL(n)$ and $K = GL(p) \times GL(q)$.  Let $X_{v_0}$ be a smooth orbit closure in $G/B$, let $I$ be a set of commuting simple reflections such that 
$I \cap \tau(v_0)$ is empty, and let $v = v_0*w_I$.  Let $\mu: Z_v : = G_{v_0} \times^B P_I/B \rightarrow X_v$ be the resolution of $X_v$ discussed above.  Then:
\begin{enumerate}
\item All fibers of $\mu$ are smooth.
 More precisely, for any good base point $\dot{y}B/B\in X_v$, the fiber $\mu^{-1}(\dot{y}B/B)$ is isomorphic
to $\dot{y}P_{I \cap \tau(y) }/B \cong {\mathbb P}^k$, where $k = | I \cap \tau(y) |$.  Moreover,
$\dot{y}P_{I \cap \tau(y) }/B \cong {\mathbb P}^k$
is homogeneous for
 $K \cap \dot{y} P_{I \cap \tau(y)} \dot{y}^{-1}$.

\item All fibers of $\mu$ are strongly reduced.
\end{enumerate}
\end{theorem}

\begin{proof}
(1) Write $P = P_I$.  We claim that $\mu^{-1}(K\dot{y}B/B) \subseteq K\dot{y}B \times^B P/B$.  Indeed, suppose $[z,p] \in \mu^{-1}(K\dot{y}B/B)$; then $z \in
K\dot{u}B/B$ for some $u \leq v_0$.  Since $\mu([z,p]) \in K\dot{y}B/B$, we see that $K\dot{u}P/P = K\dot{y}P/P$, so by Proposition \ref{p:orbitnumber}, $K\dot{u}B = K\dot{y}B$,
proving the claim.  Let $\gamma$ denote the restriction of $\mu$ to $K\dot{y}B \times^B P/B$.  The claim implies that $\mu^{-1}(\dot{y}B/B) = \gamma^{-1}(\dot{y}B/B)$.

We have a Cartesian diagram
$$
\begin{CD}
K \dot{y} B \times^B P @>{\tilde{\gamma}}>> K\dot{y}P \\
@VVV @VVV \\
K \dot{y} B \times^B P/B @>\gamma>> K\dot{y}P/B 
\end{CD}
$$
Both the source and target of $\tilde{\gamma}$ are homogeneous for the group $K \times P$, so all fibers of $\tilde{\gamma}$ are isomorphic to
the quotient of two stabilizer groups: $(K \times P)^{\dot{y}} / (K \times P)^{[\dot{y}, 1]}$.  In particular, they are smooth.
Since the diagram is Cartesian, all fibers of $\gamma$ are isomorphic to fibers of $\tilde{\gamma}$ so the fibers of $\gamma$ are smooth.  
In particular, $\mu^{-1}(\dot{y}B/B) = \gamma^{-1}(\dot{y}B/B)$ is smooth.  We conclude that the fiber of $\mu$ over any good base point is smooth.  By
Proposition \ref{p:fiber-general}, all fibers of $\mu$ are smooth.

Since $X_v$ is normal and $\mu$ is birational, Zariski's connectedness theorem implies that $\mu$ has connected fibers.  Since $\mu^{-1}(\dot{y}B/B) = \gamma^{-1}(\dot{y}B/B)$,
$\gamma^{-1}(\dot{y}B/B)$ is connected.
Since $\gamma^{-1}(\dot{y}B/B)$ is smooth and connected, it is irreducible.

By Proposition \ref{p:fiber-general}, there is an isomorphism
\begin{equation} \label{e:resolution}
 K\dot{y}B/B \cap \dot{y}P/B \stackrel{\cong}{\longrightarrow} \gamma^{-1}(\dot{y}B/B), \hspace{0.2in} zB/B \mapsto [z, z^{-1}\dot{y}B/B ].
\end{equation}
Hence $K\dot{y}B/B \cap \dot{y}P/B$ is smooth and irreducible.
We will now prove that
\begin{equation} \label{e.fiber}
K\dot{y}B/B \cap \dot{y}P/B = \dot{y}P_{I \cap \tau(y) }/B.
\end{equation}
First, 
if $s\in\tau(y)\cap I$, then
by Lemma \ref{l:compact-im}, $s$ is $y$-compact imaginary, so by \eqref{e:compact-imag}, $\dot{y}P_s/B\subseteq K\dot{y}B/B$.  It follows that
$\dot{y}P_{I \cap \tau(y) }/B \subseteq K\dot{y}B/B \cap \dot{y}P/B$.

Let $J = I \setminus (I \cap \tau(y) )$.  The dimension of the image of $\gamma$ is
$\ell(y * w_I) = \ell(y*w_J) = \ell(y) + |J|$, where the second equality is by Lemma \ref{l:tau2}.  Since all fibers of $\gamma$ are isomorphic,
\begin{align*}
\dim \gamma^{-1}(y) & = \dim (K \dot{y}B \times^B P/B ) - \dim K\dot{y}P/B = \ell(y) + |I | - (\ell(y) + |J |) \\
& = |I \cap \tau(y)| = \dim \dot{y}P_{I \cap \tau(y) }/B.
\end{align*}
Since $\dot{y}P_{I \cap \tau(y) }/B$ is complete and has the same dimension as $ \gamma^{-1}(\dot{y}B/B) \cong K\dot{y}B/B \cap \dot{y}P/B$, it must be an irreducible component of $K\dot{y}B/B \cap \dot{y}P/B$.
Since $K\dot{y}B/B \cap \dot{y}P/B$ is irreducible, we conclude that $K\dot{y}B/B \cap \dot{y}P/B \cong  \dot{y}P_{I \cap \tau(y) }/B$.  Hence $\gamma^{-1}(\dot{y}B/B) \cong \dot{y}P_{I \cap \tau(y) }/B$.

We now show that $\dot{y}P_{I \cap \tau(y) }/B \cong {\mathbb P}^k$ is homogeneous for $K\cap \dot{y}P \dot{y}^{-1}$.  
Let $L' =  \prod_{i \in I \cap \tau(y) } L'_i $,  where $L'_i $ is as in
Section \ref{ss.generalities-resolutions}.
Then $L' \subset P_{I \cap \tau(y)}$,
If $s_i \in I \cap \tau(y) $ then, as noted above, $s_i$ is $y$-compact imaginary.
Hence 
$ \dot{y} L' \dot{y}^{-1}  \subset K \cap \dot{y} P_{I \cap \tau(y)}\dot{y}^{-1}$, so
$yL'y^{-1}$ acts
on $K\dot{y}B/B \cap \dot{y} P/B \cong \dot{y} P_{I \cap \tau(y)}$.  Lemma \ref{l:commutingsimple} implies that $\dot{y} L' \dot{y}^{-1}$ acts transitively on
$\dot{y} P_{I \cap \tau(y)} \cong  \mu^{-1}(\dot{y}B/B)$; hence so does the group $K \cap \dot{y} P_{I \cap \tau(y)} \dot{y}^{-1}$.  This proves (1).

(2) The action of $K \cap \dot{y} P_{I \cap \tau(y)} \dot{y}^{-1}$ on $\mu^{-1}(\dot{y}B/B)$ does not lift
to an action on $Z_v$ preserving $\mu^{-1}(\dot{y}B/B)$, so the map \eqref{e:resolution} cannot be
$K \cap \dot{y} P_{I \cap \tau(y)} \dot{y}^{-1}$-equivariant.  
We will circumvent this issue below by using $\widetilde{Z}_v$ instead of $Z_v$ at the
appropriate place.

We must prove that if $\dot{y}B/B \in X_v$ (that is, $y \leq v$),
and $z \in F = \mu^{-1} (\dot{y}B/B)$, then
\begin{equation} \label{e:kernel-deriv}
\mathrm{Ker}(d\mu_z)=T_z\mu^{-1}(\dot{y}B/B).
\end{equation}
By Proposition \ref{p:fiber-general},
we may assume $\dot{y}B/B$ is a good base point, i.e., 
$y \leq v_0$.  We may also assume that $\dot{y} \in \mathcal{V}$.  
The inclusion $\mathrm{Ker}(d\mu_z) \supseteq T_z F$ holds 
because
$\mu(F) = \{ \dot{y}B/B \}$, so $d\mu (T_z F) \subseteq T_z (\{ \dot{y}B/B \}) = \{ 0 \}$.  

Consider the Cartesian diagram
$$
\begin{CD}
\tilde{Z}_v = G_{v_0} \times^B P @>{\tilde{\mu}}>> G_v \\
@VVV @VVV \\
Z_v = G_{v_0} B \times^B P/B @>\mu>> X_v. 
\end{CD}
$$
We have $\mu^{-1}(\dot{y}B/B) \cong \tilde{\mu}^{-1}(\dot{y})$.
The map $\tilde{\mu}$ is $K \times P$-equivariant, where $G \times P$ acts on $\tilde{Z}_v$ by $(k, p) \cdot [g, q] = [k g, q p^{-1}]$, and on $G_v$ by $(k,p) \cdot g = k g p^{-1}$.
Define $\psi: K\cap \dot{y}P \dot{y}^{-1} \to K \times P$
by $\psi(k) = (k, \dot{y}^{-1} k^{-1} \dot{y})$.  Then $\tilde{\mu}$ is $K\cap \dot{y}P \dot{y}^{-1}$-equivariant, and the fiber $\tilde{\mu}^{-1}(\dot{y})$ is preserved
by this action.  As in the first paragraph of the proof, we have $\tilde{\mu}^{-1}(\dot{y}) \cong \tilde{\gamma}^{-1}(\dot{y})$.  The analogue of 
\eqref{e:resolution} holds:
\begin{equation} \label{e:resolution2}
 K\dot{y}B/B \cap \dot{y}P/B \stackrel{\cong}{\longrightarrow} \tilde{\gamma}^{-1}(\dot{y}), \hspace{0.2in} zB/B \mapsto [z, z^{-1}\dot{y} ].
\end{equation}
Under this isomorphism, the action of $K\cap \dot{y}P \dot{y}^{-1}$ on $\tilde{\gamma}^{-1}(\dot{y})$
corresponds to the action on $K\dot{y}B/B \cap \dot{y}P/B \cong  \mu^{-1}(\dot{y}B/B)$ by left multiplication.  This action
is transitive (since the action of the subgroup $K\cap \dot{y}P_{I \cap \tau(y)} \dot{y}^{-1}$ is).

By Lemma \ref{l:changefiber}, the equality \eqref{e:kernel-deriv} is equivalent to
$\mathrm{Ker}(d \tilde{\mu}_{\tilde{z}})=T_{\tilde{z}} F$,
where $\tilde{F} = \tilde{\mu}^{-1}(\dot{y})$.  Since
the $K \cap \dot{y}P \dot{y}^{-1}$-action on $\tilde{Z}_v$ restricts to a transitive action on $\tilde{F}$,
there exists
$k \in K \cap \dot{y}P \dot{y}^{-1}$ such that $k \cdot \tilde{z}= [\dot{y},1]$.  
The isomorphism $dk_z: T_{\tilde{z} }\tilde{Z}_v \to T_{[\dot{y},1]} \tilde{Z}_v$ takes $T_{\tilde{z} }\tilde{F}$ to 
$T_{[\dot{y},1]} \tilde{F}$, and $\mathrm{Ker}(d\tilde{\mu}_{\tilde{z}})$ to $\mathrm{Ker}(d \tilde{\mu}_{[\dot{y},1]})$.  Therefore, we may assume $\tilde{z} = [\dot{y},1]$,
so $z = [\dot{y}, e]$, where $e$ is the point $B/B$ in $P/B$.

We will prove \eqref{e:kernel-deriv} by finding a subspace $V$ of $T_{[\dot{y},e]} Z_v$ such that
$d\mu_{[\dot{y},e]}|_V$ is injective, and such that $\dim V + \dim T_{[\dot{y},e]} F = \dim T_{[\dot{y},e]} Z_v$.  
This suffices, since $d\mu_{[\dot{y},e]}|_V$ injective forces $V \cap T_{[\dot{y},e]} F  = \{ 0 \}$.  
The dimension equality then implies that $ V \oplus T_{[\dot{y},e]} F =  T_{[\dot{y},e]} Z_v$, so
$\mathrm{Ker}(d\mu_z)=T_{[\dot{y},e]} F$, as desired.

Let $J = I \smallsetminus (I \cap \tau(y))$, and let $v' = v_0*w_J$.  We have a commutative diagram
$$
\begin{CD}
Z_{v'} = G_{v_0} \times^B P_J/B @>{\mu'}>> X_{v'} \\
@VVV @VVV \\
Z_{v} = G_{v_0} \times^B P_I/B @>{\mu}>> X_{v},
\end{CD}
$$
where the vertical maps are inclusions.  Let $V = T_{[\dot{y},e]} Z_{v'} \subset T_{[\dot{y},e]} Z_{v}$.
Since $\dim F = | I \cap \tau(y) |$, we see that $\dim Z_{v'} + \dim F = \dim Z_v$.  Since the varieties
in this equation are smooth, we have $\dim V + \dim T_{[\dot{y},e]} F = \dim T_{[\dot{y},e]} Z_v$. 

To show that $d\mu_{[\dot{y},e]}|_V$ is injective, it suffices to show that $\mu'$ restricts to an isomorphism of
a neighborhood $\mathcal Z$ of $[\dot{y},e]$ in $Z_{v'}$ with a neighborhood of $y$ in $X_{v'}$.  Part (1) of this theorem, applied
to $\mu'$, implies that $(\mu')^{-1}(\dot{y}B/B)$ is a single point, so $(\mu')^{-1}(\dot{y}B/B) = \{ [\dot{y},e] \}$.  Let $\U$ be the subset
of $X_{v'}$ where the fiber dimension is $0$, and let ${\mathcal Z} = (\mu')^{-1}(\U)$.
Since $\mu'$ is proper,
the fiber dimension is upper semi-continuous, so $\U$ is open in $X_{v'}$.
Since $X_{v'}$ and hence $\U$ is normal, and $\mu'$ is birational, Zariski's Connectedness Theorem implies that $\mu'_{\mathcal Z}$ is bijective;
then Zariski's Main Theorem implies that $\mu'_{\mathcal Z}$ is an isomorphism.
\end{proof}

Combining Proposition \ref{p:small} (which describes when the resolution $\mu$ in Theorem \ref{t:fiberhomog} is small)
 with Theorems \ref{t:char-cycle} and \ref{t:fiberhomog},
we obtain the following result.  As discussed in Section \ref{ss:small}, the hypotheses of the theorem can be stated explicitly
in terms of the entries of the clan $v_0$.

\begin{theorem} \label{t:irreducible}
Let $G = GL(n)$ and $K = GL(p) \times GL(q)$.  Let $X_{v_0}$ be a smooth orbit closure in $G/B$, and let $I$ be a set of commuting simple reflections such that 
$\ell(v_0* w_I) = \ell(v_0)+ |I|$ and $\ell(y* w_I) = \ell(y) + |I|$ for all $y \leq v_0$ with $\ell(y) = \ell(v_0) - 1$.  Then for $v=v_0*w_I$, the characteristic
cycle $\CC(\IC_{X_v})$ is irreducible.
\end{theorem}

\begin{proof}
By \cite[Prop.~5.3.2]{Lar}, under the hypotheses of the theorem, the resolution $\mu: Z_v \to X_v$ is small. The fibers of $\mu$ are smooth and
strongly reduced, by Theorem \ref{t:fiberhomog}. Therefore, by Theorem \ref{t:char-cycle}, $\CC(\IC_{X_v})$ is irreducible.
\end{proof}

Theorem \ref{t:irreducible} will allow us to compute (equivariant) Chern-Mather
classes.  
We begin with some background.  Suppose $Y$ is a $T$-variety.  
The $T$-equivariant
Borel-Moore homology groups of $Y$ are by definition $H^T_*(Y) = H_*(Y_T)$, where
the right hand side denotes the Borel-Moore homology of the mixed space $Y_T = E_T \times^T Y$.
Here $E_T \to B_T$ is an approximation to the universal $T$-bundle,
chosen appropriately corresponding to the degree of the homology group; we
refer to \cite{EdGr98}, \cite{AnFu24} for more details.  We refer to \cite{Jones2010} for
the definition of Chern-Mather classes.
The equivariant Chern-Mather class $c_M^T(Y) \in H^T_*(Y)$ corresponds to the class $c_M(Y_T) \in H_*(Y_T)$
(see \cite[\S4.3]{Ohm06}).
If $Y \subset X$ is a closed embedding of $Y$ into a smooth variety $X$ such that 
$\CC_Y$ is irreducible, and $\mu: Z \to Y$ is a small resolution, then
by \cite[Theorem 1.2.1]{Jones2010}, $c_M(Y) = \mu_*(c(TZ) \cap [Z])$.  
If these varieties have $T$-actions, and the maps are $T$-equivariant, then
applying Jones's result to the mixed spaces yields the equivariant
equation $c_M^T(Y) = \mu_*(c^T(TZ) \cap [Z]_T)$.

\begin{corollary} \label{c:CM}
With hypotheses as in Theorem \ref{t:irreducible} , the $T$-equivariant Chern-Mather class of $X_v$ is given by $c_M^T(X_v) = \mu_* (c^T (T Z_v) \cap [Z_v]_T)$.
\end{corollary}

\begin{proof}
Since the resolution $\mu$ is small and $\CC(\IC_{X_v})$ is irreducible, the result follows from the discussion above.
\end{proof}

In Section \ref{ss:Chern-Mather}, we explain how to compute the equivariant pushforward $\mu_* (c^T (T Z_v) \cap [Z_v]_T)$  using localization.  

\section{Alternative presentations and isomorphic fibers} \label{ss:isomorphic}
Our resolution of $X_v$ is of the form 
$$
Z = Z_{v_0, I} = G_{v_0} \times^B P_I/B = K \dot{v}_0 P_{\tau(v_0)} \times^B P_I/B,
$$
In previous sections, we denoted $Z$ by $Z_v$, but in this section we are considering
different presentations of $Z$, so we use the notation $Z_{v_0, I}$.
The main result of this section is the following theorem, which gives alternative presentations of $Z$.

\begin{theorem}\label{t:symmetric-mu}
Assume that the hypotheses of Theorem \ref{t:irreducible}  hold.
Let $\cN$ be a subset of $I$ consisting of $v_0$-noncompact imaginary simple reflections, and
let $I' = I \smallsetminus \cN$.  
Let $v_0' = v_0 * w_{\cN}$, and \(M= \tau(v_0') \cap\tau(v)\).
There is a commutative diagram
\begin{equation}\label{e:symmetric-mu}
\begin{tikzcd}
Z_{v_0, I} = G_{v_0} \times^B P_I/B\arrow[r,"\mu"]\arrow[d,swap,"\varphi"]&\overline{KvB}/B\\
G_{v_0'} \times^{P_M} P_{\tau(v)}/B\arrow[ur,swap,"\mu'"]
\end{tikzcd}
\end{equation}
such that \(\varphi\) is a \(K\)-equivariant isomorphism.  In particular, taking $\cN$ to be empty,
we obtain an isomorphism
\begin{equation}\label{e:symmetric-mu2}
Z_{v_0, I} \cong G_{v_0} \times^{P_M} P_{\tau(v)}/B.
\end{equation}
Moreover, for general $\cN$, we have
\begin{equation}\label{e:symmetric-mu3}
Z_{v_0, I} \cong Z_{v_0', I'}.
\end{equation}
\end{theorem}

The hypotheses on $\cN$ mean that if $s_i \in \cN$, then the entries of
$v_0$ (resp.~$v_0'$) in positions $i$ and $i+1$ are opposite signs (resp.~equal numbers).
Moreover, as discussed in the proof
of the theorem, $X_{v_0'}$ is smooth, and the pair $v_0', I' $ satisfies the
hypotheses of Theorem \ref{t:irreducible}.
We remark that $G_{v_0'} = K \dot{v}_0' P_{\tau(v_0')} = K \dot{v}_0 P_{\tau(v_0')}$ by 
Proposition \ref{p:imag-noncompact} below.

Note that our original resolution of $X_v$
corresponds to a resolution
of $G_v = \overline{K \dot{v} B}$ of the form $\overline{K \dot{v}_0 B} \times^B P_I$.  This expression makes it obvious that $G_v$
is right $P_I$-invariant.  However, $G_v$ is in fact right invariant under the action of the larger group $P_{\tau(v)}$,
which motivates writing $Z$ in this new form.

If $y \leq v$, let $F_y = \mu^{-1} (x B/B)$
where $xB \in K \dot{y}B/B$.  Since $\mu$ is $K$-equivariant, $F_y$ depends up to isomorphism only on the
orbit $K \dot{y} B/B$.  Part (2) of the next corollary reduces the number of $F_y$ which need to be considered
when analyzing the fibers of $\mu$.  The original presentation of the resolution would only yield (2)
for $w \in W_I$.

\begin{corollary} \label{c:symmetric-mu}

$(1)$ For $p \in P_{\tau(v)}$, we have 
\begin{equation} \label{e:isom}
\mu^{-1}(\dot{y}B/B) = \mu^{-1}(\dot{y} p B/B),
\end{equation}

$(2)$ For $w \in W_{\tau(v)}$ and $y \leq v$, we have $F_y \cong F_{y * w}$.
\end{corollary}

\begin{proof}
 Part (1) follows from Propositions
\ref{p:fiber-general} and \ref{t:symmetric-mu}.  For (2), it is enough to show
$F_y \cong F_{z}$ for $s_{\alpha} \in \tau(v)$ and $z = y s_{\alpha} $.  By definition,
the orbit $K \dot{z} B/B$ is the open $K$-orbit in $K \dot{y} P_{\ga}/B$, so
$\dot{z} = \dot{y} p$ for some $p \in P_{\ga} \subset P_{\tau(v)}$.  Hence 
$F_y \cong F_z$ by (1).
\end{proof}

To prove Theorem \ref{t:symmetric-mu}, we need the following lemmas.

\begin{lemma} \label{l:tau}
Assume that the hypotheses of Theorem \ref{t:irreducible}  hold.  Then 
every element of $\tau(v)$
commutes with every element of $I$, and
\begin{equation} \label{e:tau}
\tau(v) = (\tau(v) \cap \tau(v_0)) \sqcup I.
\end{equation}
\end{lemma}

\begin{proof}
Suppose $s_{k} \in \tau(v)$.  We show that $s_{k}$ commutes
with every element of $I$.  This holds if $s_{k} \in I$, since
by hypothesis, $I$ consists of a set of commuting reflections.  Hence
we may assume $s_{k} \not\in I$.  We must show that neither $s_{{k-1}}$ nor $s_{{k+1}}$
is in $I$.  

We prove this by contradiction.  First suppose $s_{{k-1}} $ is in $I$ but $s_{{k+1}}$ is
not.  Write
$$
v_0 = a_{k-1} | a_k a_{k+1}, \hspace{.2in} v = b_{k-1} | b_k a_{k+1},
$$
with the understanding that these equations only record the entries of the clans in positions
$k-1, k, k+1$, and the bar is a reminder that $a_{k-1}$ is the at end of a block of $v_0$, 
and $a_k$ is the beginning of a block.  By hypothesis,  $v_0 * s_{{k-1}} > v_0$, so the discussion of
Section \ref{ss:K-orbit-particular} implies that  $(a_{k-1}, a_k)$ is of the form
$(a,b)$, $(\gre, a)$, $(a, \gre)$, or $(\gre, \gre')$, where $a$ and $b$ are unequal natural numbers
and $\gre, \gre'$ are opposite signs.,
In the first case, 
$b_{k-1} | b_k a_{k+1} = b | a b_{k+1}$, and
the mate of $b_k = a$ is to its left in $v$, so $s_{k} \not\in \tau(v)$, which is a contradiction.
In the second case, $v_0 = \gre | a a_{k+1}$ and $v = a | \gre a_{k+1}$ .  Since $\mu$ is a small resolution, the discussion
in Section \ref{ss:small} implies that $a_{k+1}$ cannot equal $\gre$ or $a$; hence $s_{k} \not\in \tau(v)$.
 In the third case, $v_0 = a | \gre a_{k+1}$ and $v = \gre | a a_{k+1}$, so $s_{k} \not\in \tau(v)$.
In the last case, $v_0 = \gre | \gre' a$ and $v = b | b a$, so $s_{k} \not\in \tau(v)$.  Hence, in all four cases,
we deduce that $s_{k} \not\in \tau(v)$, which is a contradiction.  We conclude that
$s_{{k-1}} \in I$, $s_{{k+1}} \not\in I$ cannot occur.

The proof that $s_{{k-1}} \not\in I$, $s_{{k+1}} \in I$ cannot occur is similar and is omitted.
Finally, suppose that both $s_{{k-1}}$ and $s_{{k+1}}$
are in $I$.   Write $v_0 = a_{k-1} | a_k a_{k+1} | a_{k+2}$ and
$v = b_{k-1} | b_k b_{k+1} | b_{k+2}$.  If $a_{k-1}$ is a number then $b_k$ is a number whose mate
is to its left; if $a_{k+2}$ is a number then $b_{k+1}$ is a number whose mate is to its right.
In either case, $s_{k} \not\in \tau(v)$, which is a contradiction.  Hence $a_{k-1}$ and $a_{k+2}$ are signs.
Since the hypotheses of Theorem \ref{t:irreducible} hold, the patterns of \eqref{e:badconfig} do not occur in $v_0$.
Thus, $a_k$ and $a_{k+1}$ are also signs, so $b_{k-1} = b_k$ and $b_{k+1} = b_{k+2}$ are numbers,
and $s_{k} \not\in \tau(v)$, which is a contradiction.  Hence we cannot have both  $s_{{k-1}}$ and $s_{{k+1}}$ in $I$.
We conclude that that neither $s_{{k-1}}$ nor $s_{{k+1}}$
is in $I$, as desired.

We now prove \eqref{e:tau}.  Since $v = v_0 * w_I$, we have $v*w_I = v$ so $\tau(v) \supseteq I$. 
It suffices to show that if $s_{k} \in \tau(v) \smallsetminus I$, then $s_{k} \in \tau(v_0)$.
By what we have already proved, neither $s_{k-1}$ nor $s_{k+1}$ is in $I$, so the $k, k+1$ entries
of $v_0$ and $v$ coincide.  If either is a sign, then these entries are both signs and are equal,
so $s_{k} \in \tau(v_0)$.  If the entries are both numbers, then the relative positions of their
mates are the same in $v$ and in $v_0$.  Hence $s_{k} \in \tau(v)$ implies $s_{k} \in \tau(v_0)$.
\end{proof}

The following lemma is well-known.

\begin{lemma} \label{l:fiber-small}
Let $\mu_1: Z_1 \to Y$ and $\mu_2: Z_2 \to Y$ be small resolutions of $Y$.  Then for any $y \in Y$,
$H^*(\pi_1^{-1}(y)) = H^*(\pi_2^{-1}(y))$.  Hence $\dim \pi_1^{-1}(y) = \dim \pi_2^{-1}(y)$.
\end{lemma}

\begin{proof}
This follows because the direct image of the constant
sheaf under a small resolution is the intersection cohomology complex of the target.
\end{proof}

\begin{proof}[Proof of Theorem \ref{t:symmetric-mu}]
We first show that $\varphi$ is an isomorphism under the assumption that $\cN$ is empty.  
Since $G_{v_0} = K \dot{v_0} P_{\tau(v_0)}$  is right-invariant under $P_M \subset P_{\tau(v_0)}$,
and $I \subset \tau(v)$, there is a morphism $\varphi: Z = G_{v_0} \times^B P_I/B \to G_{v_0} \times^{P_M} P_{\tau(v)}/B$
given by \(\varphi[g,p]=[g,p]\).  This morphism is proper since both source and target are complete.
We claim that to show that $\varphi$ is an isomorphism, it is enough to show that $\varphi$ is surjective.
Indeed, suppose $\varphi$ is surjective.  Then $\mu'$ is also a small resolution. By Lemma \ref{l:fiber-small},
over any given point in $X_v$, the fiber of $\mu$ and the fiber of $\mu'$ have the same dimension, 
so $\varphi$ is a finite morphism.
Since $\varphi$ is $K$-equivariant, it must be a covering
over the open $K$-orbit $K \dot{v} B/B$ in $X_v$.  Since these $K$-orbits are equivariantly simply connected, this implies $\varphi$ is an isomorphism over
this $K$-orbit, so $\varphi$ is birational.  Since both source and target of $\varphi$ are smooth, hence
normal, Zariski's connectedness theorem implies that the fibers of $\varphi$ are connected; then Zariski's Main Theorem implies that
$\varphi$ is an isomorphism.  This proves the claim.

Since $\mu = \mu' \circ \varphi$ and $\mu$ is birational, we have $\dim \varphi(Z) = \dim Z$.  On the
other hand,
\begin{equation} \label{e:dimension}
\dim K \dot{v}_0P_{\tau(v_0)}\times^{P_M} P_{\tau(v)}/B = \ell(v_0) + \ell(w_{\tau(v)}) - \ell(w_M) = \ell(v_0) + | I | = \dim Z.
\end{equation}
where the last equality follows from Lemma \ref{l:tau}.  Hence $\varphi$ is dominant; since $\varphi$ is proper
and $K\dot{v}_0P_{\tau(v_0)}\times^{P_M} P_{\tau(v)}/B$ is irreducible, we conclude that $\varphi$ is surjective.
This proves that $\varphi$ is an isomorphism in case $\cN$ is empty; in particular, it proves that  \eqref{e:symmetric-mu2} is an isomorphism.

We now consider the general case.  By induction on $|\cN|$, it suffices to prove the result in case $\cN = \{ s_{i} \}$,
so $v_0' = v_0 * w_{\cN} = v_0 * s_{i}$.  As noted before the proof, the only positions where
the clan of $v_0$ differs from that of $v_0'$ are in positions $i, i+1$; in $v_0$, the entries are opposite signs, and
in $v_0'$ (and in $v$), they are equal numbers.  The clan $v_0'$ has a block decomposition where
the blocks are of the form $\max(V_{p_j,q_j})$ (since $v_0$ does), so $X_{v_0'}$ is smooth.  
Since the pair $v_0, I$ avoids the four configurations of  \eqref{e:badconfig}, so does the pair $v_0', I' = I \smallsetminus \cN$,
so this pair satisfies the hypotheses of Theorem \ref{t:irreducible}.  Since $v_0\leq v_0'$,
we have $K \dot{v}_0 P_{\tau(v_0)} = \overline{K \dot{v}_0B} \subset\overline{K{v_{0}'}B} = K v_0' P_{\tau(v_0')}$.

We claim that $\tau(v_0) \cap \tau(v) \subseteq \tau(v_0') \cap \tau(v)$. Indeed, the only positions where
the clan of $v_0$ differs from that of $v_0'$ are in positions $i, i+1$, so
the only possible roots which are in $\tau(v_0)$ but not in $\tau(v_0')$
are $\ga_{i-1}$ and $\ga_{i+1}$.  Neither of these roots is in $\tau(v)$; the claim follows.

Let $M_0 = \tau(v_0) \cap \tau(v)$ and $M = \tau(v_0') \cap \tau(v)$.  We have
$$
\varphi:  Z_{v_0, I} \cong K \dot{v}_0 P_{\tau(v_0)} \times^{P_{M_0}} P_{\tau(v)}/B \to K \dot{v}_0' P_{\tau(v_0')} \times^{P_{M}} P_{\tau(v)}/B \cong Z_{v_0', I'},
$$
where the two congruences follow from the first part of the proof, and the existence of the map
in the middle follows from the discussion above.
Both $v_0, I$ and $v_0', I'$ satisfy the hypotheses of Theorem \ref{t:irreducible}, so by the first part of the proof,
both source and target of $\varphi$ are small resolutions of $X_v$.  Moreover, $\varphi$
commutes with the maps to $X_v$.   Both source and target of $\varphi$
are smooth, complete and irreducible, and the maps to $X_v$ are birational.
Hence, reasoning as in the first part of the proof
shows that $\varphi$ is an isomorphism.  

Finally, we prove that \eqref{e:symmetric-mu3} is an isomorphism.  Applying equation \eqref{e:symmetric-mu2}
to $v_0', I'$ implies
that the target of $\varphi$ in Theorem \ref{t:symmetric-mu}
is isomorphic to $Z_{v_0', I'}$.  Hence $\varphi$ gives an isomorphism
$Z_{v_0,I} \to Z_{v_0', I'}$, as desired. 
\end{proof}

\begin{remark}\label{r:symmetric-mu}
By abuse of notation, because $\varphi$ is an isomorphism, we sometimes write $\mu$ for the morphism denoted by
$\mu'$ in Theorem \ref{t:symmetric-mu}.
\end{remark}

\begin{proposition} \label{p:imag-noncompact}
Suppose $v_0$ corresponds to a smooth orbit closure, $\ga_i$ is a $v_0$-noncompact imaginary simple root,
and the hypotheses of Theorem \ref{t:irreducible}  hold for $I = \{ s_i \}$. 
Set $v_0' = v_0 * s_{i}$.  Then
\begin{equation} \label{e:para-equal}
K \dot{v}_0 P_{\tau(v_0')} = K \dot{v}_0' P_{\tau(v_0')}.
\end{equation}
\end{proposition}

\begin{proof}
The hypothesis of the proposition implies that the $i$ and $i+1$ entries of $v_0$ are opposite signs.  
The $i$ and $i+1$ entries of $v_0'$ are equal numbers, and all other entries of $v_0'$ coincide with
the corresponding entries of $v_0$.  Since $X_{v_0}$ is smooth, $v_0$ has a block decomposition
where each block is of the form $\max(V_{p_j,q_j})$.  This implies that $v_0'$ also has a block decomposition
of this form, so $X_{v_0'}$ is smooth.  Let $M = \tau(v_0) \cap \tau(v_0')$. 
(It is possible that $\tau(v_0') \not\subseteq \tau(v_0)$, e.g.,~ this holds for 
$v_0 = (+--)$ and $i = 1$.)
 
Lemma \ref{l:tau} (with $v_0'$ playing the role of $v$)  implies that $w_{\tau(v_0')} = s_{i} w_M $.  
Moreover, $w_M$ and $s_{i}$ commute and satisfy
$\ell(w_M) + \ell(s_{i}) = \ell(w_{\tau(v_0')})$.  Therefore,
$v_0 * w_{\tau(v_0')} = v_0'$, since
$$
v_0 * w_{\tau(v_0')} = v_0 * s_{i} * w_M = v_0 * s_{i} * s_{i} * w_M = v_0' * s_{i} * w_M = v_0' *  w_{\tau(v_0')} = v_0'.
$$
Letting $\pi: G/B \to G/P_{\tau(v_0')}$ be the natural map, we have 
$$
K \dot{v}_0 P_{\tau(v_0')}/B = \pi^{-1} (K \dot{v}_0 P_{\tau(v_0')}/P_{\tau(v_0')}) = \overline{K (v_0 * w_{\tau(v_0')}) B}/B = \overline{K \dot{v}_0' B}/B = K \dot{v}_0' P_{\tau(v_0')}/B,
$$
as desired.
\end{proof}

\begin{example}
Let $G = GL(2)$, $K = GL(1) \times GL(1)$, and \(v_0=(+-)\).
Set $I=\left\{s_1\right\}  = \cN$, so \(w_I=s_{1}\), \(v=v_0\star s_{1}=(11)\), and $\tau(v) = \{ s_{1} \}$.
Then $G_{v_0} = B$, $M = I$, $P_M = G$,
and the map $\varphi$ is the natural isomorphism
$B \times^B G/B \to G \times^G G/B$.
\end{example}

\begin{example}
Let $G = GL(7)$, $K = GL(4) \times GL(3)$, $v_0 = (+|-|+|1221)$, $I = \{s_1,s_3 \}$, $\cN = \{s_1 \}$.  Then
$v_0' = (33|+|1221)$ and $v = (33|1+221)$.  We have $\tau(v_0) = \{ s_4,s_5,s_6 \}$, $\tau(v_0') = \{ s_1,s_4,s_5,s_6 \}$,
$\tau(v) = \{ s_1,s_3,s_5,s_6 \}$, and $M = \tau(v_0) \cap \tau(v) = \{s_1,s_5,s_6 \}$.  From
Theorem \ref{t:symmetric-mu} and Proposition \ref{p:imag-noncompact}, we obtain 
$$
G_{v_0} \times^B P_{1,3}/B  \stackrel{\varphi}{\longrightarrow}   G_{v_0'} \times^{P_{1,5,6}} P_{1,3,5,6}/B  
=  K \dot{v}_0 P_{1,4,5,6} \times^{P_{1,5,6}} P_{1,3,5,6}/B.
$$
\end{example}

\begin{corollary}\label{corollary: symmetric mu}
Assume the same hypotheses and notation as in Proposition~\ref{t:symmetric-mu}.
The map \(\pi\,\colon G_{v_0'} /P_M\to G_v/P_{\tau(v)}\) obtained by restricting the projection \(G/P_M\to G/P_{\tau(v)}\) is a small resolution.
\begin{proof}
Using the isomorphism $G_v \times^{P_{\tau(v)}} P_{\tau(v)}/B \cong G_v/B$, we see
that the following diagram is Cartesian:
\begin{equation}
\begin{tikzcd}
G_{v_0'} \times^{P_M} P_{\tau(v)}/B\arrow[r,"\mu"]\arrow[d]& G_v/B\arrow[d]\\
G_{v_0'} /P_M\arrow[r,"\pi"]&G_v /P_{\tau(v)}.
\end{tikzcd}
\end{equation}
Here the vertical arrows are smooth maps with fiber \(P_{\tau(v)}/B\).
\end{proof}
\end{corollary}

\section{Tangent spaces and Chern-Mather classes} \label{s:tangent}
We obtain formulas for Chern-Mather classes by applying
localization in equivariant cohomology, together with
the description of tangent spaces to
mixed spaces at $T$-fixed points given in
Corollary \ref{c:tangent2}.  This corollary is proved in
Section \ref{ss:mixed},
which slightly generalizes 
\cite[Section 3]{GZ}.  Section \ref{ss:tangentK} applies these
results to resolutions of $K$-orbit closures.  Most of the results of this subsection apply for arbitrary
$G$ and $K = G^{\theta}_0$, but at the end of the subsection we specialize to $G = GL(n)$ and $K = GL(p) \times GL(q)$.
Section \ref{ss:fixedpoint-gl} describes the torus-fixed points in 
$Z_v$ and $X_v$ for the resolutions we consider, and analyzes
which fixed points in $Z_v$ lie in the fiber over a fixed point in $X_v$.

\subsection{Tangent spaces to mixed spaces} \label{ss:mixed}
The main goal of this section is to prove Corollary \ref{c:tangent2}.  We begin with some generalities.
Let $G$ be a linear algebraic group with closed subgroups $R$ and $S$, and let $M$ be a smooth algebraic variety with
an $R$-action.  Let $Q$ be a closed subvariety of $G$ which is left $S$-invariant and right $R$-invariant, and such that $Q/R$ is a smooth algebraic variety.
Let $X = Q \times^R M$ and let $[q_0, m_0]$ be an $S$-fixed point of $X$.
Then $q_0 R/R \in Q/R$ is also $S$-fixed, so the tangent spaces $T_{[q_0, m_0]} X$ and $T_{q_0 R} (Q/R)$ are representations of $S$.  If $Q$ is a subgroup of $G$,
this situation was analyzed in \cite[Section 3]{GZ}.  Here we require greater generality, since we will take
$Q = G_{v_0} = \overline{K v_0 B}$.

Since $[q_0, m_0]$ is $S$-fixed, $q_0 R/R \in Q/R$ is fixed by $S$, i.e., $q_0^{-1} S q_0 \subseteq R$,
and $m_0$ is fixed by $q_0^{-1} S q_0$.  If $\xi \in \fr$, define vector fields $\xi^{\dagger}$ on $Q$ and $\xi^{\#}$ on $M$ by the rules
\begin{equation}
\xi^{\dagger}_q f_1 = \frac{d}{dt} f_1(q \exp (t \xi)) \Big|_{t=0}, \hspace{.3in} \xi^{\#}_m f_2  = \frac{d}{dt} f_2(\exp (t \xi) m) \Big|_{t=0} 
\end{equation}
for $q\in Q$ and $m\in M$, where $f_1$ and $f_2$ are functions on $Q$ and $M$, respectively.
Define $\psi_1: \fr \to T_{q_0} Q$ and $\psi_2: \fr \to T_{m_0}M$ by $\psi_1(\xi) = {\xi}^{\dagger}_{q_0}$ and
$\psi_2(\xi) = \xi^{\#}_{m_0}$.  Define $\psi = (\psi_1, - \psi_2) : \fr \to V : = T_{q_0} Q \oplus T_{m_0}M \cong T_{(q_0, m_0)} (Q \times M)$.

We now define several $S$-actions.
Let $S$ act on $\fr$ by 
$s \cdot \xi = \Ad (q_0^{-1} s q_0) \xi$.  
Let $S$ act on $Q$ by $s*q = L(s) \circ R(q_0^{-1} s^{-1} q_0) q = s q q_0^{-1} s^{-1} q_0$. 
Similarly, let $S$ act on $M$ by $s*m = L(q_0^{-1} s q_0) m = q_0^{-1} s q_0 m$. 
The direct product of these actions yields an $S$-action (again denoted by $*$) on $Q \times M$.
Since $(q_0,m_0)$ is $S$-fixed, we obtain an $S$-action on $T_{(q_0, m_0)} Q \times M$.
Explicitly, this is the direct product of the $S$-action
on $T_{q_0}(Q)$, where $s \in S$ acts as $L(s)_* R(q_0^{-1} s^{-1} q_0)_*$, with
the $S$-action on $T_{m_0}M$, where $s$ acts as $L(q_0^{-1} s q_0)_*$.

\begin{lemma} \label{l:equivariant}
The maps $\pi: Q \times M \to Q \times^R M$ and $\pi_1: Q \to Q/R$ are $S$-equivariant, where $S$ acts by the $*$-action on the source
and by left multiplication on the target.
\end{lemma}

\begin{proof}
Since $q_0 R$ is $S$-fixed, if $s \in S$, then $s q_0 = q_0 r$ for some $r \in R$.  Therefore
\begin{align*}
\pi(s*(q,m)) & = \pi(s q q_0^{-1} s q_0, q_0^{-1} s^{-1} q_0 m) = [s q q_0^{-1} s q_0, q_0^{-1} s^{-1}  q_0 m] \\
& = [s q q_0^{-1} q_0 r, q_0^{-1} s^{-1}  q_0 m] =  [s q , r q_0^{-1} s^{-1}  q_0 m] \\
& = [sq, (q_0^{-1} s q_0) q_0^{-1} s^{-1}  q_0 m] = [sq, m].
\end{align*}
Hence $\pi$ is $S$-equivariant.  The proof that $\pi_1$ is $S$-equivariant is similar but easier, and we omit it.
\end{proof}

\begin{proposition} \label{p:exact}
\begin{enumerate}
\item The maps $\psi_1: \fr \to T_{q_0} Q$ and $\psi_2: \fr \to T_{m_0}M$ are $S$-equivariant.
\item The maps $(\pi_1)_*: T_{q_0} Q \to T_{q_0 R/R} (Q/R)$ and $\pi_*: V \to T_{[q_0, m_0]} X$ are $S$-equivariant.
\item We have exact sequences of $S$-modules:
\begin{equation} \label{e:exact1}
\begin{CD}
0 @>>> \fr @>{\psi_1}>> T_{q_0} Q @>{(\pi_1)_*}>> T_{q_0 R/R} (Q/R) @>>> 0
\end{CD}
\end{equation}
and
\begin{equation} \label{e:exact2}
\begin{CD}
0 @>>> \fr @>{\psi}>> T_{q_0} Q \oplus T_{m_0} M @>{\pi_*}>> T_{[q_0,m_0]} X @>>> 0.
\end{CD}
\end{equation}
\end{enumerate}
\end{proposition}

\begin{proof}
(1) The equivariance of $\psi_2$ follows from \cite[Lemma 3.1]{GZ}.  To verify that $\psi_1$ is $S$-equivariant, we need to check that for $s \in S$ and $\xi \in \fr$,
\begin{equation} \label{e:exact3}
({\Ad(q_0^{-1} s q_0) \xi})^{\dagger}_{q_0} = L(s)_* R(q_0^{-1} s q_0)_* (\xi^{\dagger}_{q_0}).
\end{equation}
We calculate the result by applying each side to a function $f$.  The left hand side yields
\begin{align*}
({\Ad(q_0^{-1} s q_0) \xi})^{\dagger}_{q_0} f & = \frac{d}{dt} f( q_0 \exp (t \Ad(q_0^{-1} s q_0) ) \xi) \Big|_{t=0} \\
& = \frac{d}{dt} f( q_0 q_0^{-1} s q_0 \exp(t \xi) q_0^{-1} s^{-1} q_0 )  \Big|_{t=0} \\
& = \frac{d}{dt} f( s q_0 \exp(t \xi) q_0^{-1} s^{-1} q_0 )  \Big|_{t=0} 
\end{align*}
The right hand side yields
\begin{align*}
L(s)_* R(q_0^{-1} s q_0)_* (\xi^{\dagger}_{q_0}) f & = \xi^{\dagger}_{q_0} (f \circ L(s) \circ R(q_0^{-1} s q_0) ) \\
 & =  \frac{d}{dt} f \circ L(s) \circ R(q_0^{-1} s q_0) (q_0 \exp(t \xi))  \Big|_{t=0} \\
 & =  \frac{d}{dt} f( s q_0 \exp(t \xi) q_0^{-1} s^{-1} q_0 )  \Big|_{t=0} .
\end{align*}
This proves \eqref{e:exact3}.  

(2) By Lemma \ref{l:equivariant}, the maps $\pi_1$ and $\pi$ are $S$-equivariant, so the corresponding derivative maps $(\pi_1)_*$ and $\pi_*$ are $S$-equivariant as well.

(3) By (1) and (2), the maps in the sequences are $S$-equivariant.  The map $\psi_1$ is injective because $R$ acts with trivial stabilizers on $G$, hence on
$Q$; the injectivity of $\psi$ follows from the injectivity of $\psi_1$.  The maps $(\pi_1)_*$ and $\pi_*$ are surjective because the
 $\pi_1$ and $\pi$ are principal $R$-bundle maps.  To check exactness in the middle, it suffices for dimension reasons 
 to show that  $(\pi_1)_* \circ \psi_1 = 0$ and $\pi_* \circ \psi = 0$.  We will verify the second equation; the verification of the first is similar.

For any $\xi \in \fr$, if $f$ is a function on $X$, we have
\begin{align*}
( \pi_* \circ \psi (\xi)) f & = \psi(\xi) (f \circ \pi) = (\xi^{\dagger}_{q_0}, - \xi^{\#}_{m_0} ) (f \circ \pi) =  \frac{d}{dt} f \circ \pi (q_0 \exp (t \xi), \exp (-t \xi) m_0)  \Big|_{t=0} \\
& =  \frac{d}{dt} f([q_0 \exp (t \xi), \exp (-t \xi) m_0]) \Big|_{t=0} =  \frac{d}{dt} f([q_0 , m_0]) \Big|_{t=0} = 0.
\end{align*}
 Hence $\pi_* \circ \psi (\xi) = 0$.  Since this is true for any $\xi \in \fr$, we have $\pi_* \circ \psi = 0$. 
\end{proof}

\begin{corollary} \label{c:tangent}
Keep the hypotheses above, and assume in addition that $S$ is reductive.
Then as representations of $S$, 
\begin{equation} \label{e:tangent}
T_{[q_0, m_0]} X \cong T_{q_0 R/R} (Q/R) \oplus T_{m_0} M,
\end{equation}
where $s \in S$ acts on $T_{m_0} M$ by $L(q_0^{-1} s q_0)_*$.
\end{corollary}

\begin{proof}
Because $S$ is reductive, the exact sequences \eqref{e:exact1} and \eqref{e:exact2} split.  Therefore, as $S$-representations,
$$
\fr \oplus T_{[q_0,m_0]} X \cong T_{q_0} Q \oplus T_{m_0} M \cong \fr \oplus T_{q_0 R/R} (Q/R) \oplus T_{m_0} M.
$$
Hence every irreducible $S$-representation
occurs with the same multiplicity in $T_{[q_0,m_0]} X$ and $T_{q_0 R/R} (Q/R) \oplus T_{m_0} M$.  
\end{proof}

If $S$ is a torus normalized by $q_0$,
then $m_0$ is fixed by $S$, and $T_{m_0} M$ is a representation of $S$ via the map $s \mapsto L(s)_*$.
In the next corollary, $\Phi(T_{m_0} M)$ refers to the weights of $T_{m_0}M$ with respect to this action.

\begin{corollary} \label{c:tangent2}
Keep the hypotheses above, and assume in addition that $S$ is a torus normalized by $q_0$.  Then as multisets,
$$
\Phi(T_{[q_0, m_0]} X) = \Phi(T_{q_0 R/R} (Q/R)) \sqcup  q_0 \Phi(T_{m_0} M).
$$
\end{corollary}

\begin{proof}
This follows from Corollary \ref{c:tangent}, since $q_0 \Phi(T_{m_0} M)$ is the set of weights of the $S$-action on $T_{m_0} M$ 
used in that corollary.
\end{proof}

\begin{remark} \label{r:GZ}
Proposition \ref{p:exact} and Corollary \ref{c:tangent} are analogous (respectively) to Proposition 3.4 and Corollary 3.5 of \cite{GZ}.
In that paper, $Q$ was assumed to be a group, and the Lie algebra $\fq$ of $Q$ was used in place of the tangent space $T_{q_0}Q$.
If $Q$ is a group, then the map $L(q_0)_*: \fq \to T_{q_0}Q$ is an $S$-equivariant isomorphism, where
$h \in S$ acts on $\fq$ via $\Ad(q_0^{-1} s q_0)$, and on $T_{q_0}Q$ by $L(s)_* \circ R(q_0^{-1} s^{-1} q_0)_*$.
This isomorphism can be used to rewrite Proposition \ref{p:exact} and Corollary \ref{c:tangent} with $\fq$ in place
of $T_{q_0}Q$, thus recovering the results of \cite{GZ}.  Similarly, Corollary \ref{c:tangent2} is analogous to \cite[Cor.~3.11]{GZ}.
\end{remark}

\subsection{Tangent spaces to resolutions of $K$-orbit closures} \label{ss:tangentK}
In this section we describe the tangent spaces of the iterated bundles used to
construct resolutions of
$K$-orbit closures as in Section \ref{s:resolution}.
At the end of the subsection we specialize to $G = GL(n)$ and $K = GL(p) \times GL(q)$ and the resolutions $Z_v \to X_v$.

With notation as in Section \ref{s:resolution}, let
$Z$ be as in \eqref{e:generalZ}.
Although \(T\) acts on \(Z\), the group \(H\) need not act.  
We will describe \(Z^T\) in terms of \(W\) (and see that $Z^T$ is finite), and describe the multisets
 $\Phi(T_z Z)$ for $z \in Z^T$.  Although the $T$-weights in \(T_zZ\) can
occur with multiplicity greater than $1$, we will see that each \(T\)-weight is nonzero, i.e., every 
$T$-fixed point is nondegenerate.

For each $k$, let $W^{P_{J_k}}$ denote a set of coset representatives
in $W$ for $W/W_{P_k}$; for definiteness, we will take minimal coset representatives.
Identify $W^{P_{J_k}}$ with $G/P_{J_k}$
via the map $x \mapsto \dot{x} P_{J_k}$, where $\dot{x} \in N_G(H) $ is a representative for $x \in W = N_G(H)/H$.

Let $x_0, \ldots, x_m$ be elements of
$W$ such that for each $k$, $\dot{x}_k \in N_G(H) \cap G_{u_{k}}$ is a representative for $x_k$ in $N_G(H)$.
Then  $x = [\dot{x}_0, \ldots, \dot{x}_m]$ is in $Z^T$.  This element does not depend on the choice
of $\dot{x}_k$, so by abuse of notation, write $x = [x_0, \ldots, x_m]$.  We
claim that by changing the $x_k$ if necessary, we may assume that $x_{k-1} \in W^{P_{J_k}}$ for each $k$.  Indeed, we can find
$y_0 \in P_{J_1}$ such that $x_0 y_0 \in W^{P_{J_1}}$.  Then $[x_0, \ldots, x_m] = [x_0 y_0, y_0^{-1} x_1, x_2, \ldots, x_m]$, so by replacing
$x_0$ by $x_0 y_0$ and $x_1$ by $ y_0^{-1} x_1$, we may assume that $x_0 \in W^{P_{J_1}}$.  Iterating this argument, we
can modify $x_2, \ldots, x_m$ in succession so that $x_k \in  W^{P_{J_k}}$, proving the claim.

Recall that if $V$ is a representation of $T$ then $\Phi(V,T)$ is the multiset of $T$-weights of $V$, counted with multiplicity.

\begin{proposition}\label{p:T-fixed-weights}
Let $Z$ be as in \eqref{e:generalZ}.
Then
\begin{equation} \label{e:H-fixed1}
Z^T =\left\{ [x_0,\ldots,x_m] \mid  x_k  \in (G_{u_k}/P_{J_{k+1}})^T \;\text{for all}\;\; 0\leq k\leq m\right\} \subset W^{J_1} \times \cdots \times W^{J_{m+1}}.
\end{equation}
If $Z$ is smooth and $x = [x_0, \ldots x_m] \in Z^T$, then
\begin{equation} \label{e:H-fixed2}
\Phi(T_xZ) = \Phi_0 \sqcup x_0 \Phi_1 \sqcup \cdots \sqcup x_0 x_1 \cdots x_{m-1} \Phi_m,
\end{equation}
where $\Phi_k = \Phi( T_{{x_k}}(G_{u_k}/P_{J_{k+1}}),T)$ and \(\Phi(T_xZ)  = \Phi(T_xZ,T)\).   
\end{proposition}

\begin{proof}
By the discussion before the proposition, when we write $[x_0,\ldots,x_m]$ (where $ x_k  \in (G_{u_k}/P_{J_{k+1}})^T$ for each $k$),
we mean $[\dot{x}_0, \ldots, \dot{x}_m]$, where $\dot{x}_k \in N_G(H) \cap G_{u_k}$.  A straightforward calculation
shows that such an element of $Z$ is fixed by $T$.  Therefore, to prove \eqref{e:H-fixed1}, we must show that if
$x \in Z^T$, then
\begin{equation} \label{e:H-fixed3}
x = [x_0,\ldots,x_m] \mbox{  where  } \dot{x}_k \in N_G(H) \cap G_{u_k} \mbox{  for all  }k.
\end{equation}

Suppose then that $x = [g_0, \ldots, g_m] \in Z^T$, where $g_k \in G_{u_k}$ for each $k$.  Then $g_0 P_{J_1} \in (G/P_{J_1})^T = (G/P_{J_1})^H$,
so $g_0 P_1 = \dot{x}_0 P_1$ for some $\dot{x}_0 \in N_G(H)$.  Thus, $g_0 = \dot{x}_0 p_1$ for $p_1 \in P_{J_1}$.  Since 
$g_0 \in G_{u_0}$ and $G_{u_0}$ is right
$P_{J_1}$-invariant, $\dot{x}_0 \in G_{u_0}$.  We have
$x = [\dot{x}_0 p_1, g_1, \ldots, g_m] = [\dot{x}_0, p_1 g_1, g_2, \ldots, g_m]$.  
Hence
$[p_1 g_1, g_2, \ldots, g_m] \in G_{u_1} \times^{P_{J_2}} \cdots \times^{P_{J_m}} G_{u_m}/P_{J_{m+1}}$ is
fixed by $\dot{x}_0^{-1} T \dot{x}_0$, so $p_1 g_1 \in (G/P_{J_2})^{\dot{x}_0^{-1} T \dot{x}_0} =  (G/P_{J_2})^{\dot{x}_0^{-1} H \dot{x}_0}
= (G/P_{J_2})^H$.  Reasoning as before shows that $p_1 g_1 = \dot{x}_1 p_2$ where $\dot{x}_1 \in N_G(H) \cap G_{u_1}$, and
$ [\dot{x}_0, p_1 g_1, g_2, \ldots, g_m] = [\dot{x}_0, \dot{x}_1, p_2 g_2, \ldots, g_m $.  Equation \eqref{e:H-fixed3}
follows by iterating this argument, so
\eqref{e:H-fixed1} holds.  Equation \eqref{e:H-fixed2} then
follows by an iterated application of Corollary~\ref{c:tangent2}.
\end{proof}

\begin{remark} \label{r:T-fixed-weights}
If $Z$ is as in \eqref{e:generalZ-A2} (which is a special case of \eqref{e:generalZ})., then the description of
$\Phi_k$ can be made more precise: $\Phi_0 = \Phi(\fk/(\fk \cap{}^{x_0}\fp_{I_0})$, and 
for $k \geq 1$, $\Phi_k = \Phi(\fp_{I_k}/{}^{x_k}\fp_{J_{k+1}})$, where $I_0 = \tau(v_0)$,, and
$\fp_J = \mbox{Lie }P_J$ for $J \subset S$.  If $H = T$, the description
of tangent spaces can be reformulated in a way suited for
computer calculations such as those used in Section \ref{s:examples}.
Let $\langle \ \cdot \ , \ \cdot \ \rangle$ denote the natural $\mathbb{Z}$-valued
pairing between characters and cocharacters of $T$.
Choose dominant cocharacters \(\gamma_k\,\colon\C^\times\to T\) 
such that \(\left\{s\in S\mid \langle\alpha_s\vert_T,\gamma_k\rangle=0\right\}=J_k\), and \(\chi_k\,\colon\C^\times\to T\) 
such that \(\left\{s\in S\mid \langle \alpha_s\vert_T,\chi_k\rangle=0\right\}=I_k\).
Let \(\Phi_K = \Phi( \fk) \subset \Phi_G = \Phi( \fg) \).  Then \eqref{e:H-fixed2} yields
\begin{equation}\label{e:atlas-weights}
\begin{split}
\Phi(T_xZ)=&\left\{\alpha\in\Phi_K\mid\langle\alpha, x_0\gamma_1\rangle<0\right\} \sqcup \left\{\alpha\in\Phi_G\mid\langle\alpha, x_0\chi_1\rangle\geq0,\ \langle\alpha,x_0x_1\gamma_2\rangle<0\right\}\\
&\sqcup \left\{\alpha\in\Phi_G\mid \langle\alpha, x_0x_1\chi_2\rangle\geq0,\ \langle\alpha, x_0x_1x_2\gamma_3\rangle<0\right\} \sqcup \cdots\\
&\sqcup \left\{\alpha\in\Phi_G\mid \langle\alpha, x_0\cdots x_{m-1}\chi_m\rangle\geq 0,\ \langle\alpha, x_0\cdots x_{m}\gamma_{m+1}\rangle<0\right\}.
\end{split}
\end{equation}
\end{remark}

We now specialize to the family of resolutions in Theorem~\ref{t:irreducible}.  Let $Z = Z_v = G_{v_0} \times^B P_I/B$.  To describe 
$Z^T$, we use the identification $Z = G_{v_0} \times^{P_{\tau(v_0)}} P_{\tau(v_0)} \times^B P_I/B$, and recall from
Proposition \ref{p:fixedpoint-smooth} the identification of $(G_{v_0}/P_{\tau(v_0)})^T$ with $W(v_0) \cap W^{\tau(v_0)}$.

\begin{corollary}\label{c:T-weights}
Let $Z  = Z_v =  G_{v_0} \times^{P_{\tau(v_0)}} P_{\tau(v_0)} \times^{P_{\tau(v_0)}} P_I/B$.  Then
$$
Z^T = \{ [x_0, x_1, x_2] \mid x_0 \in ( W(v_0) \cap W^{\tau(v_0)} ), x_1 \in W_{\tau(v_0)}, x_2 \in W_I \}.
$$
Given $x = [x_0, x_1, x_2]$, we have
$$
\Phi(T_x Z) = \Phi(\fk/(\fk\cap{}^{x_0}\fp_{\tau(v_0)}))\cup x_0 x_1\Phi(\fp_{\tau(v_0)}/\fb)\cup x_0 x_1 x_2 \Phi(\fp_I/\fb).
$$
\end{corollary}

\begin{proof}
The description of $T$-fixed points follows from Proposition \ref{p:T-fixed-weights} and 
the identification of $(G_{v_0}/P_{\tau(v_0)})^T $ with $W(v_0) \cap W^{\tau(v_0)}$.
Since ${}^{x_1} \fp_{\tau(v_0)} = \fp_{\tau(v_0)}$, we have $\fp_{\tau(v_0)}/ {}^{x_1}\fb = {}^{x_1} \fp_{\tau(v_0)}/ {}^{x_1}\fb$,
so $\Phi(\fp_{\tau(v_0)}/ {}^{x_1}\fb) = x_1\Phi(\fp_{\tau(v_0)}/\fb)$.  Similarly, $\Phi(\fp_I/  {}^{x_2}\fb) = x_2 \Phi(\fp_I/\fb)$.
The description of $\Phi(T_x Z)$ follows from these observations and Remark \ref{r:T-fixed-weights}.
\end{proof}

\subsection{Fixed points and fibers} \label{ss:fixedpoint-gl}
Throughout this section we assume $G = GL(n)$ and $K = GL(p) \times GL(q)$, and
we work with the resolution $\mu: Z_v = G_{v_0} \times^B P_I/B \to X_v$.
In this section we describe the $T$-fixed points lying in a fiber of $\mu$
over a $T$-fixed point.  As a consequence, we show that there are
closed $K$-orbits contained in the subset of $X_v$ over which $\mu$ is birational.

Given $y \in W$, let $I_y$ denote the set of $ s_i \in I$ such that $y_i$ and $y_{i+1}$ are 
either both in $[p]$ or both in $[n] \smallsetminus [p]$.

\begin{proposition} \label{p:fixedpoint-inverse}
We have
\begin{equation} \label{e:fixedpoint-inverse1}
X_v^T = \{ y B/B \mid y = wx \mbox{ for some } w \in W(v_0), x \in W_I \}.
\end{equation}
For such $y$,
\begin{equation} \label{e:fixedpoint-inverse2}
\mu^{-1}(yB/B)^T = \{ [w z, z^{-1} x B/B ]  \mid z \in W_{I_y} \}.
\end{equation}
Hence $| (\mu^{-1}(yB/B))^T | = | W_{I_y} |$.
\end{proposition}

\begin{proof}
Since $X_v^T = \mu(Z_v^T)$, \eqref{e:fixedpoint-inverse1} follows from Corollary \ref{c:T-weights}.  

Suppose $\mu([w,xB/B]) = y B/B$, and let
$\xi$ be in $(\mu^{-1}(\dot{y}B/B))^T$.
We have
$\xi = [y_1, y_2 B/B]$ for some $y_1 \in W$, $y_2 \in W_I$ with $y_1 y_2 = y$.  Since $wx = y$, we
see that $y_1 = w z$ where $z = x y_2^{-1} \in W_I$; then $y_2 = z^{-1}x$, so $\xi = [w z, z^{-1} x B/B]$.

Since $z \in W_I$, 
$z$ is a product of reflections $s_{i} \in I$.  Multiplying $w$ on the right by $s_i$ switches the entries $w_i$ and $w_{i+1}$, which are in
adjacent blocks of $w$.  Since both $wB/B$ and $w z B/B$ are in
$X_{v_0}$, each block of $w$ and each block of $wz$ must satisfy the condition of Proposition \ref{p:fixedpoint-smooth}.  This
forces  $w_i$ and $w_{i+1}$ to both be in $[p]$ or to both be in $[n] \smallsetminus [p]$, so $s_i \in I_w = I_y$, where the second equality is because
$\{w_i, w_{i+1} \} = \{y_i, y_{i+1} \} $.  Hence $z \in W_{I_y}$.  

Reversing this reasoning shows that if $z \in W_{I_y}$, then
$[w z, z^{-1} x B/B ] \in (\mu^{-1}(\dot{y}B/B))^T $.  This proves \eqref{e:fixedpoint-inverse1}.  
For distinct $z$, the elements $[w z, z^{-1} x B/B ] $ are distinct, so 
$| (\mu^{-1}(yB/B))^T | = | W_{I_y} |$.
\end{proof}

\begin{example}
Let $G = GL(7) \supset K = GL(4) \times GL(3)$.  Let $v_0 = (1 + 1 | 2 2 | 3 3)$. and $I = \{ s_3, s_5 \}$, so $v = v_0 * w_I = (1 + 2 1 3 2 3)$. 
Let $y = (3 4 6 | 7 1 | 2 5)$.  Then $yB/B \in X_{v_0} \subset X_v$.  In this case, $I_y = I$, and there are $4$ elements in $\mu^{-1}(yB/B)^T$:
$$
\mu^{-1}(yB/B)^T = \{ [y, eB/B],[y s_3, s_3 B/B],  [y s_5, s_5 B/B],  [ys_3 s_5, s_3 s_5 B/B] \}.
$$
On the other hand, if $y = (3 6 4 | 7 1 | 5 2) \in X_{v_0}$, then $I_y$ is empty, and $(\mu^{-1}(yB/B))^T = \{ [y, eB/B] \}$.  
If $y = (3 6 7 | 4 1 | 5 2)$, then $yB/B$ is in $X_v$ but not $X_{v_0}$; $I_y$ is empty, and $(\mu^{-1}(yB/B))^T = \{ [(y s_1, s_1 B/B)] \}$.
\end{example}

We define the birational locus of $\mu: Z_v \to X_v$ to be the open subset $U$ of $X_v$ such
that $\mu|_{\mu^{-1}(U)} \to U$ is an isomorphism.  Since $Z_v$ and $X_v$ are normal, the birational locus is simply
the set of $p \in X_v$ such that $\mu^{-1}(p)$ consists of a single point (see \cite[AG.18]{Borel1991}).

\begin{corollary} \label{c:inverseimage}
\begin{enumerate}
\item If $gB \in X_v$ is in a closed $K$-orbit, then $\mu^{-1}(gB) \cong (\mathbb{P}^1)^k$, where $yB/B \in (K gB/B)^T$
and $k = | I_y |$.
\item Suppose $v_0 = (a_1, \ldots, a_n)$, and suppose that for each $s_{\ga_i} \in I$, the
blocks of $v_0$ containing $a_i$ and $a_{i+1}$ each contain at least two numbers or two opposite signs.
Then there exist closed $K$-orbits in $X_v$ contained in the birational locus of $\mu$.
\end{enumerate}
\end{corollary}

\begin{proof}
Since any closed $K$-orbit contains a $T$-fixed point, there exists $y \in W$ such that $yB \in K gB/B$.  
Since the fibers $\mu$ over a $K$-orbit are isomorphic, $\mu^{-1}(gB) \cong \mu^{-1}(yB)$.
By Theorem \ref{t:fiberhomog} and Proposition \ref{p:fiber-general}, $\mu^{-1}(yB/B)$ is isomorphic to $(\mathbb{P}^1)^k$ for some $k$.
The order of $\mu^{-1}(yB/B)^T$ is equal to the Euler characteristic of $(\mathbb{P}^1)^k$, which is $2^k$.
On the other hand, Proposition \ref{p:fixedpoint-inverse} implies that $| \mu^{-1}(yB/B)^T | = | W_{I_y} | = 2^{| I_y |}$.  This proves (1).

For (2), given $v_0$, it suffices to find a closed orbit $X_u$ such that for each $s_{i} \in I$, the entries $u_i$ and $u_{i+1}$ are
signs and are equal, since then $I_y$ is empty for $yB/B \in X_u$, so $\mu^{-1}(yB/B)$ is a single point.  
We can find $u$ as follows.  Let $ I = \{ s_{{i_1}}, \ldots, s_{{i_r}} \}$.  Then let $u_{i_1} = +$, $u_{i_1+1} = +$,
$u_{i_2} = -$, $u_{i_2+1} = -$, $u_{i_3} = +$, $u_{i_3+1} = +$, etc.  The process assigns at most one $+$ sign and one $-$ sign to
each block of $u$.  The $j$-th block of $v_0$ equals $\max(V_{p_j, q_j})$; our hypothesis on
$v_0$ ensures that $p_j \geq 1$ and $q_j \geq 1$.
The remaining entries of $u$ can
be chosen in any manner such that the $j$-th block has $p_j$ $+$ signs and $q_j$ $-$ signs.
\end{proof}

\begin{example} \label{e:closed-birational}
The hypothesis of part (2) of the preceding proposition
is sufficient but not necessary for the birational locus of $\mu$ to contain a closed orbit.
Let $G = GL(5)$, $K = GL(3) \times GL(2)$, $I = \{s_{1}, s_{4} \}$.
If $v_0 = (+ | 1-1 | +)$, then there is no closed orbit contained in the birational locus of
$\mu$.  The reason is that for any closed orbit $X_u$ in $X_v$,
at least one of  $(u_1, u_2)$ or $(u_4, u_4)$ must be
a pair of opposite signs, so $| I_y | \geq 2$ for $yB \in X_u$.  On the other hand, if $v_0 = (- | 1+1 | +)$, then the closed
orbit corresponding to $u = (- | -++ | +)$ is in the birational locus.
\end{example}

\subsection{Push-forwards of Chern classes and Chern-Mather classes} \label{ss:Chern-Mather}
By Corollary \ref{c:CM}, the $T$-equivariant Chern-Mather class of $X_v$ is given by $c_M^T(X_v) = \mu_* (c^T (T Z_v) \cap [Z_v]_T)$.
Localization in equivariant cohomology and Borel-Moore homology can be used to compute the right hand side.
More generally, if $\mu: Z \to Y$ is a $T$-equivariant map of $T$-varieties where $Z$ is nonsingular and $Z^T$ is finite,
we explain how to calculate $\mu_*(c^T(TZ) \cap [Z]_T) \in H^T_*(Y)$.

The $T$-equivariant Borel-Moore homology groups are a module for $H^*_T$, the $T$-equivariant cohomology of a point.  For convenience,
we use complex coefficients; then $H^*_T$ is isomorphic to the symmetric algebra $S(\ft^*)$, where
$\ft = \mbox{Lie } T$.  Let $\cq$ denote the
quotient field of $H^*_T$.  

Let $q \in Z^T$, and let $i_q: \{q \} \hookrightarrow Z$ denote the inclusion.  The tangent space $T_q Z$ is a representation of
$T$; suppose it has weights $\gb_1, \ldots, \gb_n$.  The weights are all nonzero, but they are not necessarily distinct.
Then the $n$-th Chern class and total Chern class of $T_q Z$ are elements of $H^*_T$ given, respectively, by
$c_n^T(T_q Z) = \gb_1 \gb_2 \cdots \gb_n$ and $c^T(T_q Z) = (1 + \gb_1) (1 + \gb_2) \cdots (1 + \gb_n)$.
We have $i_q^* c^T_n(TZ) = c_n^T(T_q Z)$ and $i_q^* c^T(TZ) = c^T(T_q Z)$.  Define
\begin{equation} \label{e:cq}
C(q) = \frac{c^T(T_q Z)}{c^T_n(T_q Z)} = \frac{\prod \gb_i}{\prod(1 + \gb_i) } \in \cq.
\end{equation}

\begin{proposition} \label{p:pushforward}
In $H^T_*Y \otimes_{H^*_T} \cq$, we have
\begin{equation} \label{e:pushforward}
\mu_* (c^T(TZ) \cap [Z]_T) = \sum_{p \in Y^T} ( \sum_{q \in (\mu^{-1}(p))^T} C(q) ) [p].
\end{equation}
\end{proposition}

\begin{proof}
The localization theorem implies that in $H^T_*Z \otimes_{H^*_T} \cq$,
\begin{align*}
c^T(TZ) \cap [Z]_T & = \sum_{q \in Z^T}  i_{q*}  \Big( \frac{1}{c^T_n(T_q Z)} i_q^{*}  (c^T(TZ) \cap [Z]_T) \Big) \\
& = \sum_{q \in Z^T}  i_{q*} (\frac{c^T(T_q Z)}{c_n^T(T_q Z)} \cap [q]) = \sum_{q \in Z^T}  i_{q*} (C(q) \cap [q]).
\end{align*}
We abuse notation and write $ i_{q*} [q] = [q]$, where on the left side, $[q]$ denotes the fundamental class in
$H_*^T(\{ q \})$, and on the right side, $[q]$ is viewed as a class in $H_*^T(Z)$.  With these conventions,
$ i_{q*} (C(q) \cap [q]) = C(q) [q]$, and so, in $H^T_*Z \otimes_{H^*_T} \cq$, we have
$$
c^T(TZ) \cap [Z]_T = \sum_{q \in Z^T} C(q) [q]
$$
Therefore, in $H^T_*Y \otimes_{H^*_T} \cq$, we have
\begin{align*}
\mu_* (c^T(TZ) \cap [Z]_T)  & = \sum_{q \in Z^T} \mu_*(C(q) [q]) = \sum_{q \in Z^T} C(q) \mu_* [q] \\
& = \sum_{p \in Y^T} \Big( \sum_{q \in (\mu^{-1}(p))^T} C(q) \Big) [p],
\end{align*}
proving the result.
\end{proof}

\begin{remark} \label{r:eq-mult}
By the localization theorem, $[X]_T = \sum_{p \in X^T} \varepsilon_p(X) [p]_T$
and $[Z]_T = \sum_{q \in Z^T} \varepsilon_q(Z) [q]_T$.  The elements $\varepsilon_p(X)$ and $\varepsilon_q(Z)$ of $\cq$
are called equivariant multiplicities (see \cite{Brion1997})).
Since $Z$ is smooth, $\varepsilon_q(Z) = \prod_{\gb_i \in \Phi(T_q Z)} \gb_i$.  If $\mu$ is birational, then 
$\varepsilon_p(X) = \sum_{q \in (\mu^{-1}(p))^T} \varepsilon_q(Z)$.
\end{remark}

We now return to the $K$-orbit closure $X_v$.  We abuse
notation and view $c_M^T(X_v)$ as an element of $H^T_*(X)$
by taking its image under the map $H^T_*(X_v) \to H^T_*(X)$.
We can write $c_M^T(X_v) = \sum b_y [yB]_T$, where, in light
of Corollary \ref{c:CM}, formulas for the coefficients $b_y$ can
be obtained by
applying Corollary \ref{c:T-weights} and Proposition \ref{p:pushforward} to our resolution $\mu: Z_v \to X_v$.  
Given $w \in W$, let
$Y_w = \overline{BwB}/B$ denote the corresponding Schubert
variety.  The classes $[Y_w]_T$ form an $H^*_T$-basis
of $H^T_*(X)$, so
\[
[yB]_T = \sum_{y \in W} m_{yw} [Y_w]_T.
\] 
The $m_{yw}$ can be computed
using a formula due Andersen-Jantzen-Soergel and Billey (see \cite{AJS94}, \cite{Billey99})  or a formula due
to Kumar (see \cite{Kumar1996}). 
Hence we can calculate the coefficient of $[Y_w]_T$ in the expression
\begin{equation} \label{e:CM-Schubert}
c_M^T(X_v) = \sum_{w \in W} c^v_w [Y_w]_T = \sum_{w \in W} \Big( \sum_{y \in W} b_y m_{yw} \Big) [Y_w]_T
\end{equation}
and therefore express $c_M^T(X_v)$ in terms of the Schubert basis $[Y_w]_T$.
A similar change of basis was applied in the setting of Springer fibers in  \cite[Section 4]{GZ}, to which we refer
for a more detailed discussion.

\section{Examples} \label{s:examples}
In this section we explain how our resolutions can be realized
as configuration spaces, and explain this realization in two examples.
In the first example, we calculate the equivariant Chern-Mather
class and verify our positivity conjecture in an example (see Section \ref{ss:1212}).  
(the Chern-Mather class calculation does not use the configuration space realization).
In the second example, we show how the configuration space realization can be used
to verify our description of the fibers.

\subsection{Configuration space realizations} \label{ss:configuration}
Resolutions like the ones studied in this paper were realized as configuration spaces in 
\cite{Lar}, building on previous work such as \cite{GelfandMacPherson1982} and \cite[\S5.2]{EWY}.  The construction is as follows.
As usual, we identify $G/B$ with the variety of complete flags of subspaces of $\C^n$.
Write $E^\bullet = E^1 \subset E^2 \subset \cdots \subset E^n$ for such a flag, where
$\dim E^i = i$.  There is an isomorphism $Z_v = G_{v_0} \times^B P_I/B \to X_{v_0} \times_{G/P_I} G/B$ given by $[g,pB] \mapsto (gB, gpP)$.
The target consists of pairs of flags $(E^\bullet, F^\bullet)$, where $E^\bullet \in X_{v_0}$, and 
$E^k = F^k$ whenever $s_k \not\in I$.  In terms of the fiber product realization, $\mu$ is projection on
the second factor, so the fiber $\mu^{-1}(F^{\bullet})$ depends on the possibilities 
for $E^i$ for $s_i \in I$.  These possibilities are limited by the conditions that $E^\bullet \in X_{v_0}$,
and $E^k = F^k$ for $s_k \not\in I$.

The conditions for $E^\bullet$ to be in $X_{v_0}$ are described in \cite[Corollary 1.3]{Wys}.
Given $Z = Z_v$ we depict the conditions for $(E^\bullet, F^\bullet)$ to be in $Z$ in a diagram where
the height of a subspace is its dimension, and a line denotes an inclusion.  For $s_k \not\in I$, we 
write only $F^k$ in the diagram (since
$E^k = F^k$).
 Let $e_{\ell}$ denote the $\ell$-th standard
basis element of $\C^n$.  Write $\C^n = \C^p \oplus \C^{-q}$ for the decomposition into the $+1$ and $-1$
eigenspaces of the matrix $\begin{bmatrix} I_p & 0 \\ 0 & - I_q \end{bmatrix}$, so $\C^p$ is the span of $e_1, \ldots, e_p$
and $\C^{-q}$ is the span of $e_{p+1}, \ldots, e_n$.

\subsection{Example: $v=(1212)$} \label{ss:1212}
Let $G = GL(4)$, $K = GL(2) \times GL(2)$.  Let \(v_0=(1\;1\;2\;2)\), $I = \{ s_2 \}$, and \(v=(1\;2\;1\;2)\).  
In this section we give the configuration space realization of the resolution $\mu: Z = Z_v = G_{v_0} \times^B P_I/B\to X_v$,
and use this to verify the assertion in Theorem \ref{t:fiberhomog} that the fibers are
products of copies of $\mathbb P^1$.  We also use this resolution to compute \(c_M^T(X_v)\) ,
and verify Conjecture \ref{conj:positive}.  Note that in this example,
$\tau(v_0) = \{s_1, s_3 \}$,  $\tau(v) = I$, and $\tau(v_0) \cap \tau(v)$ is empty, so Theorem \ref{t:symmetric-mu} does not yield 
an alternative form of the resolution.

By \cite[Corollary 1.3]{Wys}, a flag $E^\bullet$ is in
$X_{v_0}$ if and only if $\dim A^1 = \dim B^1 = 1$, where
$A^1 = \C^2 \cap E^2$ and $B^1 = \C^{-2} \cap E^2$.  
The realization of $Z_v$ as a configuration space is depicted in Figure \ref{Fig3}.  
\begin{figure}[H]
\begin{equation}
\begin{tikzcd}
&&\C^4\\
&&F^3\arrow[u,dash]\\
\C^{2}\arrow[uurr,dash]&\C^{-2}\arrow[uur,dash]&E^2\arrow[u,dash]&F^2\arrow[ul,dash]\arrow[dl,dash]\\
A^1\arrow[urr,gray,dash]\arrow[u,dash]\arrow[drr,dash]&B^1\arrow[ur,dash]\arrow[dr,dash]\arrow[u,dash]&F^1\arrow[u,dash]\\
&&0\arrow[u,dash]
\end{tikzcd}
\end{equation}
\captionof{figure}{Configuration space of $Z_v$ for $v = (1212)$}
\label{Fig3}
\end{figure}
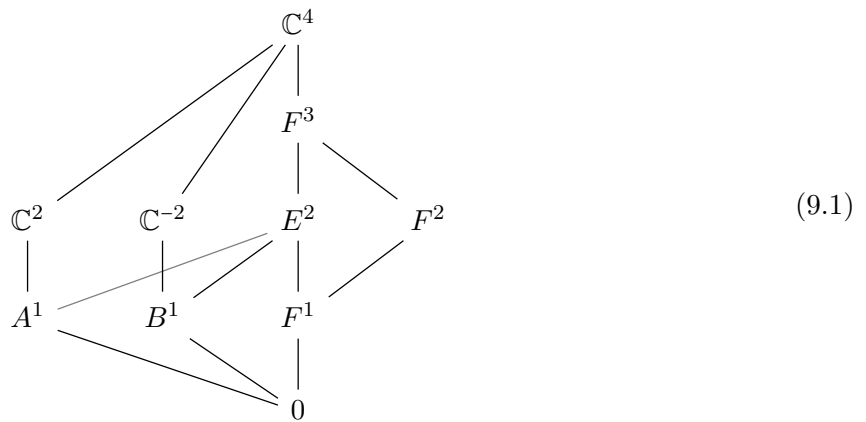

From Section \ref{ss:tangentK}, 
we can write $Z$ in the form
\begin{equation} \label{e:alternative1212}
Z=Kv_0P_{1,3}\times^{P_{1,3}}P_{1,3}\times^{B}P_{2}/B.
\end{equation}
By Corollary \ref{c:T-weights},
\begin{equation} \label{e:zt-example}
Z^T = \{ [x_0, x_1, x_2] \mid x_0 \in W(v_0) \cap W^{\tau(v_0)}, x_1 \in W_{1,3}, x_2 \in W_2\}.
\end{equation}
Here
\begin{equation} \label{e:fp}
W(v_0) \cap W^{\tau(v_0)} = \set{(13|24),(14|23),(23|14),(24|13)}.
\end{equation}
We have $\mu([x_0, x_1, x_2]) = x_0 x_1 x_2 B/B$.  Concretely, $x_0$ is one of the permutations on the right
hand side of \eqref{e:fp}; multiplication by $x_1$ can switch either or both (or neither) of the first two or last two entries,
and multiplication by $x_2$ can switch the middle two entries.  Every \(w\in W\) can be obtained by this process, so
every \(T\)-fixed point in \(G/B\) is in $X_v$.  This also follows since $(G/B)^T$ is contained in the set of closed
$K$-orbits, and by inspecting the closure order on $K$-orbits depicted in \cite{Wys}, we can see that
$X_v$ contains every closed $K$-orbit.

Following the method of Section \ref{ss:Chern-Mather}, we computed the expansion \eqref{e:CM-Schubert}
by using Proposition \ref{p:pushforward} and the weights 
$\Phi(T_z Z_v)$ to compute the $b_y$, and by using the formula of \cite{AJS94} and \cite{Billey99}
to compute
the  \(m_{yw}\).
For example, we found
\begin{equation}
\begin{split}
c_v^{4321}=&\alpha_{1}^{3} \alpha_{2}^{2} \alpha_{3} + \alpha_{1}^{3} \alpha_{2}^{2} + \alpha_{1}^{3} \alpha_{2} \alpha_{3}^{2} + 3 \alpha_{1}^{3} \alpha_{2} \alpha_{3} + 2 \alpha_{1}^{3} \alpha_{2} + \alpha_{1}^{3} \alpha_{3}^{2} + 2 \alpha_{1}^{3} \alpha_{3} + \alpha_{1}^{3} + 2 \alpha_{1}^{2} \alpha_{2}^{3} \alpha_{3} \\& + 2 \alpha_{1}^{2} \alpha_{2}^{3} + 3 \alpha_{1}^{2} \alpha_{2}^{2} \alpha_{3}^{2} + 9 \alpha_{1}^{2} \alpha_{2}^{2} \alpha_{3} + 6 \alpha_{1}^{2} \alpha_{2}^{2} + \alpha_{1}^{2} \alpha_{2} \alpha_{3}^{3} + 7 \alpha_{1}^{2} \alpha_{2} \alpha_{3}^{2} + 12 \alpha_{1}^{2} \alpha_{2} \alpha_{3} \\& + 6 \alpha_{1}^{2} \alpha_{2} + \alpha_{1}^{2} \alpha_{3}^{3} + 4 \alpha_{1}^{2} \alpha_{3}^{2} + 5 \alpha_{1}^{2} \alpha_{3} + 2 \alpha_{1}^{2} + \alpha_{1} \alpha_{2}^{4} \alpha_{3} + \alpha_{1} \alpha_{2}^{4} + 2 \alpha_{1} \alpha_{2}^{3} \alpha_{3}^{2} \\& + 7 \alpha_{1} \alpha_{2}^{3} \alpha_{3} + 5 \alpha_{1} \alpha_{2}^{3} + \alpha_{1} \alpha_{2}^{2} \alpha_{3}^{3} + 9 \alpha_{1} \alpha_{2}^{2} \alpha_{3}^{2} + 16 \alpha_{1} \alpha_{2}^{2} \alpha_{3} + 8 \alpha_{1} \alpha_{2}^{2} + 3 \alpha_{1} \alpha_{2} \alpha_{3}^{3} \\& + 12 \alpha_{1} \alpha_{2} \alpha_{3}^{2} + 14 \alpha_{1} \alpha_{2} \alpha_{3} + 5 \alpha_{1} \alpha_{2} + 2 \alpha_{1} \alpha_{3}^{3} + 5 \alpha_{1} \alpha_{3}^{2} + 4 \alpha_{1} \alpha_{3} + \alpha_{1} + \alpha_{2}^{4} \alpha_{3} \\& + \alpha_{2}^{4} + 2 \alpha_{2}^{3} \alpha_{3}^{2} + 5 \alpha_{2}^{3} \alpha_{3} + 3 \alpha_{2}^{3} + \alpha_{2}^{2} \alpha_{3}^{3} + 6 \alpha_{2}^{2} \alpha_{3}^{2} + 8 \alpha_{2}^{2} \alpha_{3} + 3 \alpha_{2}^{2} + 2 \alpha_{2} \alpha_{3}^{3} + 6 \alpha_{2} \alpha_{3}^{2} \\& + 5 \alpha_{2} \alpha_{3} + \alpha_{2} + \alpha_{3}^{3} + 2 \alpha_{3}^{2} + \alpha_{3}
\end{split}
\end{equation}
in terms of the simple roots \(\alpha_1, \alpha_2, \alpha_3\).
All of the coefficients \(c_v^w\) were found to be sums of monomials in the positive simple roots with nonnegative integer
coefficients, confirming Conjecture \ref{conj:positive} in this example.

\subsection{Example: $v=(1+2132-3)$}
Let \(G = GL(8)\) and $K = GL(4) \times GL(4)$.  Let \(v_0=(1+1| 22 | 3-3)\), \(I=\left\{s_3,s_5\right\}   = v_0=(v_0^1, v_0^2 ,v_0^3)\), 
where \(v_0^1=\max(V_{2,1})\), \(v_0^2=\max(V_{1,1})\), and \(v_0^3=\max(V_{1,2})\).  Let $I = \left\{ s_3,s_5\right\} $; then
$v = v_0*w_I$.  
We have \(\tau(v_0)=\left\{s_1,s_2,s_4,s_6,s_7\right\}\)
and \(\tau(v)=\left\{s_1,s_3,s_5,s_7\right\}\).  The conditions of Theorem \ref{t:fiberhomog} are satisfied, so $\mu: Z_v \to X_v$ is a small resolution
with smooth and strongly reduced fibers.  Let \(M = \tau(v_0) \cap \tau(v) = \left\{s_1,s_7\right\}\); then by Theorem \ref{t:symmetric-mu},
$$
Z = G_{v_0} \times^B P_{3,5}/B \stackrel{\varphi}{\longrightarrow} G_{v_0} \times^{P_{1,7}} P_{1,3,5,7}/B \stackrel{\mu'}{\longrightarrow} X_v,
$$
where $\varphi$ is an isomorphism and $\mu = \mu' \circ \varphi$.  The description of \(Z\) as a configuration space 
is given in Figure \ref{Fig2}.
\begin{figure}[h]
\begin{minipage}[t][9cm]{.63\textwidth}
\vfill
\begin{equation*}
\begin{array}{lll}
y&F_y^\bullet&\mu^{-1}(F_y^\bullet)\\[2pt]
\hline
\\[-1em]
(11++--22)&(e_1+e_5,e_1,e_2,e_3,e_6,e_7,e_4+e_8,e_4)&\P^1\times\P^1\\
(11++-+--)&(e_1+e_5,e_1,e_2,e_3,e_6,e_4,e_7,e_8)&\P^1\\
(++-+--11)&(e_1,e_2,e_5,e_3,e_6,e_7,e_4+e_8,e_4)&\P^1\\
(11+-+-22)&(e_1+e_5,e_1,e_2,e_6,e_3,e_7,e_4+e_8,e_4)&pt\\
(++-+-+--)&(e_1,e_2,e_5,e_3,e_6,e_4,e_7,e_8)&pt\\
(11+-++--)&(e_1+e_5,e_1,e_2,e_6,e_3,e_4,e_7,e_8)&\P^1\\
(++--+-11)&(e_1,e_2,e_5,e_6,e_3,e_7,e_4+e_8,e_4)&\P^1\\
(++--++--)&(e_1,e_2,e_5,e_6,e_3,e_4,e_7,e_8)&\P^1\times\P^1\\
(11++-2-2)&(e_1+e_5,e_1,e_2,e_3,e_6,e_4+e_7,e_8,e_4)&\P^1\\
(1+1+--22)&(e_1+e_5,e_2,e_1,e_3,e_6,e_7,e_4+e_8,e_4)&\P^1\\
(11+22+--)&(e_1+e_5,e_1,e_2,e_3+e_6,e_3,e_4,e_7,e_8)&pt\\
(++-11-22)&(e_1,e_2,e_5,e_3+e_6,e_3,e_7,e_4+e_8,e_4)&pt\\
(1+1-++--)&(e_1+e_5,e_2,e_1,e_6,e_3,e_4,e_7,e_8)&\P^1\\
(++--+1-1)&(e_1,e_2,e_5,e_6,e_3,e_4+e_7,e_8,e_4)&\P^1\\
\end{array}
\end{equation*}
\caption{Fibers over some good basepoints $y\leq v_0$}
\label{table:fibers}
\end{minipage}
\hfill
\begin{minipage}[t][13cm]{0.35\textwidth}
\begin{equation*}
\begin{tikzcd}
&&\C^8\arrow[d,dash]&&&\\
&&F^7\arrow[d,dash]&&&&\\
&&F^6\arrow[d,dash]\arrow[dr,dash]&&&\\
&&E^5\arrow[d,dash]&F^5\arrow[dl,dash]&&\\
\C^4\arrow[uuuurr,dash]&\C^{-4}\arrow[uuuur,dash]&F^4\arrow[d,dash]\arrow[dr,dash]&&&&\\
A^3\arrow[u,dash]\arrow[uurr,dash,color=gray,bend right=20]&&E^3\arrow[d,dash]&F^3\arrow[dl,dash]&&\\
A^2\arrow[u,dash]\arrow[urr,dash,color=gray]\arrow[ddrr,dash,bend right=30]&B^2\arrow[uuur,dash,color=gray]\arrow[uu,dash]&F^2\arrow[d,dash]&&&\\
&B^1\arrow[uur,dash,color=gray]\arrow[u,dash]\arrow[dr,dash]&F^1\arrow[d,dash]&&&\\
&&0&&&
\end{tikzcd}
\end{equation*}
\captionof{figure}{Configuration space of $Z_v$}
\label{Fig2}
\end{minipage}
\end{figure}

In this diagram,  \(A^3=\C^4\cap E^5\), \(A^2=\C^4\cap E^3\), \(B^2=\C^{-4}\cap E^5\), and \(B^1=\C^{-4}\cap E^3\).
The conditions for $E^\bullet \in X_{v_0}$ imply that $\dim A^k = k$ and $\dim B^k = k$.  We see that
 \(E^3=A^2\oplus B^1\) and \(E^5=A^3\oplus B^2\).  By the discussion above, the
 fibers of \(\mu\) are determined by \((E^3,E^5)\).

Table \ref{table:fibers} shows fibers over certain good base points.
 The fiber over a point depends on the orbit $K \dot{y}B/B$ for \(y\leq v_0\) in which the point lies.
For each such $y$, a result from \cite{Yam} (see \cite[Section 2.1]{Wys})
gives us a representative point  \(F^\bullet_y  \in K\dot{y}B/B\).  In the table, we specify $F^\bullet_y$ by writing
a basis \((f_1,\ldots,f_8)\) of \(\C^8\) and setting \(F^k_y=\langle f_1,\ldots,f_k\rangle\).

By Theorem \ref{t:fiberhomog}, writing $k = | I \cap \tau(y) |$, the fiber
 $\mu^{-1}(F_y^\bullet)$ is isomorphic to $(\P^1)^k$, as recorded in the third column
of the table.
 The isomorphisms $\mu^{-1}(F_y^\bullet) \cong (\P^1)^k$
 can be verified directly from the description of $Z$ as a configuration space.
 For example,
 taking $y = (11++--22)$ (the first entry in the table), we see that $| I \cap \tau(y)| = | I | = 2$
 so Theorem \ref{t:fiberhomog} implies that  $\mu^{-1}(F_y^\bullet) \cong \P^1\times\P^1$.  To
 see this from the
 configuration space description, first note that the possibilities for $E^3 = A^2 \oplus B^2$ can
 be determined using $F_y^2 \subset E^3 \subset F_y^4$.  Indeed, we have
 $F^2 \cap \C^{-4} =  \langle e_5 \rangle     \subseteq B^1 \subseteq F^4 \cap \C^{-4}  = \langle e_5 \rangle$.  
 This forces $B^1 = \langle e_5 \rangle$.  Similarly,
 $F_y^2 \cap \C^{4} = \langle e_1 \rangle  \subseteq A_2 \subseteq F_y^4 \cap \C^{4} = \langle e_1,e_2, e_3 \rangle$.
 Thus, the possibilties for $A^2$ are parametrized by
 $\P((F_y^4 \cap \C^{4})/(F_y^2 \cap \C^{4}) )= \P(\langle e_1,e_2, e_3 \rangle/\langle e_1 \rangle ) \cong \P^1$.
 Similarly, the possibilities for $E^5 = A^3 \oplus B^2$ can be determined using
 $F_y^4 \subset E^5 \subset F_y^6$.  We find that $A^3 = \langle e_1, e_2, e_3 \rangle$, and the
 possibilities for $B^2$ are parametrized by 
 $\P((F_y^6 \cap \C^{-4})/(F_y^4 \cap \C^{-4})) = \P(\langle e_5, e_6, e_7 \rangle / \langle e_5 \rangle ) \cong \P^1$.
 Hence
 $$
 \mu^{-1}(F_y^\bullet) \cong \P((F_y^4 \cap \C^{4})/(F_y^2 \cap \C^{4}) ) \times \P((F_y^6 \cap \C^{-4})/(F_y^4 \cap \C^{-4})) \cong \P^1 \times \P^1,
 $$
as expected.

The preceding table describes the fibers over certain good base points.  The fibers over all other good base
points are points.  Indeed, by inspecting the closure order for orbits
contained in $X_{v_0}$ depicted in Figure \ref{figure: 1+1223-3}, one sees that
any $y \leq v_0$ not listed in the table satisfies $X_y \supset X_z$ for some $z \leq v_0$ listed in Table \ref{table:fibers}
such that $F_z$ is a point.  Since $F_y$ is connected and $\dim F_y \le \dim F_z$ by semicontinuity,
$F_y$ is a point.  In particular $F_{v_0}$ is a point.  The fibers over base points that are not necessarily good can be
determined using the isomorphism $F_y \cong F_{y*w}$ for $w \in W_{\tau(v)}$.  In particular, we see that 
$F_v$ is a point.  Since $\mu$ is surjective and $K$-equivariant and each fiber over the open $K$-orbit in $X_v$ is a point,
we recover the fact that $\mu$ is birational.

\begin{figure}[p]
\centering
\includegraphics[angle=90,width=\textwidth,keepaspectratio]{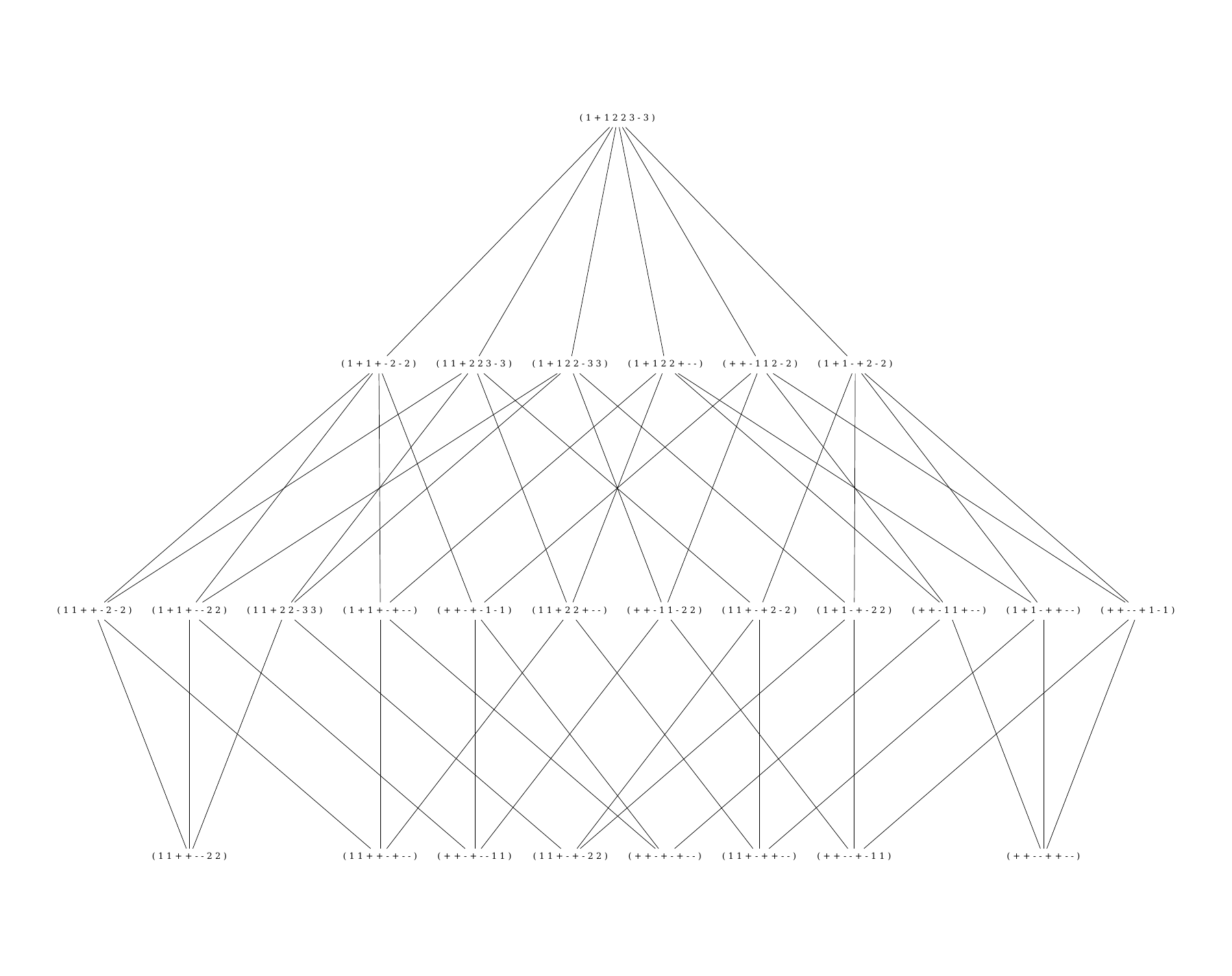}
\caption{Closure order for $(1+1223-3)$}
\label{figure: 1+1223-3}
\end{figure}

By Corollary \ref{c:T-weights}, 
$$
Z^T \cong  ( W(v_0) \cap W^{\tau(v_0)} ) \times W_{\tau(v_0)} \times W_I \cong ( W(v_0) \cap W^{\tau(v_0)} )  \times (S_3 \times S_2 \times S_3) \times (S_2)^2.
$$
An element $w \in W(v_0) \cap W^{\tau(v_0)}$ is determined by the two entries from $[4]$ in the first block and the
entry from $[8] \smallsetminus [4]$ in the first block, and the entries in the second block, one of which is from $[4]$ and one from $[8]\smallsetminus[4]$.
This gives
$$
| W(v_0) \cap W^{\tau(v_0)} | = \left( \binom{4}{2} \cdot 4 \right) \times 3 \times 2 = 144.
$$
Since $|S_m|=m!$, we conclude that
$$
|Z^T| = 144 * 6*2*6*2^2 = 41,472
$$

\bibliographystyle{amsplain}

\end{document}